\documentclass[reqno, 10pt]{amsart}
\usepackage{amssymb,amsmath,amsthm,amsfonts,color} 
\usepackage{etoolbox}

\makeatletter
\let\old@tocline\@tocline
\let\section@tocline\@tocline
\newcommand{\subsection@dotsep}{4.5}
\newcommand{\subsubsection@dotsep}{4.5}
\patchcmd{\@tocline}
  {\hfil}
  {\nobreak
     \leaders\hbox{$\m@th
        \mkern \subsection@dotsep mu\hbox{.}\mkern \subsection@dotsep mu$}\hfill
     \nobreak}{}{}
\let\subsection@tocline\@tocline
\let\@tocline\old@tocline

\patchcmd{\@tocline}
  {\hfil}
  {\nobreak
     \leaders\hbox{$\m@th
        \mkern \subsubsection@dotsep mu\hbox{.}\mkern \subsubsection@dotsep mu$}\hfill
     \nobreak}{}{}
\let\subsubsection@tocline\@tocline
\let\@tocline\old@tocline

\let\old@l@subsection\l@subsection
\let\old@l@subsubsection\l@subsubsection

\def\@tocwriteb#1#2#3{%
  \begingroup
    \@xp\def\csname #2@tocline\endcsname##1##2##3##4##5##6{%
      \ifnum##1>\c@tocdepth
      \else \sbox\z@{##5\let\indentlabel\@tochangmeasure##6}\fi}%
    \csname l@#2\endcsname{#1{\csname#2name\endcsname}{\@secnumber}{}}%
  \endgroup
  \addcontentsline{toc}{#2}%
    {\protect#1{\csname#2name\endcsname}{\@secnumber}{#3}}}%

\newlength{\@tocsectionindent}
\newlength{\@tocsubsectionindent}
\newlength{\@tocsubsubsectionindent}
\newlength{\@tocsectionnumwidth}
\newlength{\@tocsubsectionnumwidth}
\newlength{\@tocsubsubsectionnumwidth}
\newcommand{\settocsectionnumwidth}[1]{\setlength{\@tocsectionnumwidth}{#1}}
\newcommand{\settocsubsectionnumwidth}[1]{\setlength{\@tocsubsectionnumwidth}{#1}}
\newcommand{\settocsubsubsectionnumwidth}[1]{\setlength{\@tocsubsubsectionnumwidth}{#1}}
\newcommand{\settocsectionindent}[1]{\setlength{\@tocsectionindent}{#1}}
\newcommand{\settocsubsectionindent}[1]{\setlength{\@tocsubsectionindent}{#1}}
\newcommand{\settocsubsubsectionindent}[1]{\setlength{\@tocsubsubsectionindent}{#1}}

\renewcommand{\l@section}{\section@tocline{1}{\@tocsectionvskip}{\@tocsectionindent}{}{\@tocsectionformat}}%
\renewcommand{\l@subsection}{\subsection@tocline{1}{\@tocsubsectionvskip}{\@tocsubsectionindent}{}{\@tocsubsectionformat}}%
\renewcommand{\l@subsubsection}{\subsubsection@tocline{1}{\@tocsubsubsectionvskip}{\@tocsubsubsectionindent}{}{\@tocsubsubsectionformat}}%
\newcommand{\@tocsectionformat}{}
\newcommand{\@tocsubsectionformat}{}
\newcommand{\@tocsubsubsectionformat}{}
\expandafter\def\csname toc@1format\endcsname{\@tocsectionformat}
\expandafter\def\csname toc@2format\endcsname{\@tocsubsectionformat}
\expandafter\def\csname toc@3format\endcsname{\@tocsubsubsectionformat}
\newcommand{\settocsectionformat}[1]{\renewcommand{\@tocsectionformat}{#1}}
\newcommand{\settocsubsectionformat}[1]{\renewcommand{\@tocsubsectionformat}{#1}}
\newcommand{\settocsubsubsectionformat}[1]{\renewcommand{\@tocsubsubsectionformat}{#1}}
\newlength{\@tocsectionvskip}
\newcommand{\settocsectionvskip}[1]{\setlength{\@tocsectionvskip}{#1}}
\newlength{\@tocsubsectionvskip}
\newcommand{\settocsubsectionvskip}[1]{\setlength{\@tocsubsectionvskip}{#1}}
\newlength{\@tocsubsubsectionvskip}
\newcommand{\settocsubsubsectionvskip}[1]{\setlength{\@tocsubsubsectionvskip}{#1}}

\patchcmd{\tocsection}{\indentlabel}{\makebox[\@tocsectionnumwidth][l]}{}{}
\patchcmd{\tocsubsection}{\indentlabel}{\makebox[\@tocsubsectionnumwidth][l]}{}{}
\patchcmd{\tocsubsubsection}{\indentlabel}{\makebox[\@tocsubsubsectionnumwidth][l]}{}{}

\newcommand{\@sectypepnumformat}{}
\renewcommand{\contentsline}[1]{%
  \expandafter\let\expandafter\@sectypepnumformat\csname @toc#1pnumformat\endcsname%
  \csname l@#1\endcsname}
\newcommand{\@tocsectionpnumformat}{}
\newcommand{\@tocsubsectionpnumformat}{}
\newcommand{\@tocsubsubsectionpnumformat}{}
\newcommand{\setsectionpnumformat}[1]{\renewcommand{\@tocsectionpnumformat}{#1}}
\newcommand{\setsubsectionpnumformat}[1]{\renewcommand{\@tocsubsectionpnumformat}{#1}}
\newcommand{\setsubsubsectionpnumformat}[1]{\renewcommand{\@tocsubsubsectionpnumformat}{#1}}
\renewcommand{\@tocpagenum}[1]{%
  \hfill {\mdseries\@sectypepnumformat #1}}

\let\oldappendix\appendix
\renewcommand{\appendix}{%
  \leavevmode\oldappendix%
  \addtocontents{toc}{%
    \protect\settowidth{\protect\@tocsectionnumwidth}{\protect\@tocsectionformat\sectionname\space}%
    \protect\addtolength{\protect\@tocsectionnumwidth}{2em}}%
}
\makeatother

\makeatletter
\settocsectionnumwidth{2em}
\settocsubsectionnumwidth{2.5em}
\settocsubsubsectionnumwidth{3em}
\settocsectionindent{1pc}%
\settocsubsectionindent{\dimexpr\@tocsectionindent+\@tocsectionnumwidth}%
\settocsubsubsectionindent{\dimexpr\@tocsubsectionindent+\@tocsubsectionnumwidth}%
\makeatother

\settocsectionvskip{10pt}
\settocsubsectionvskip{0pt}
\settocsubsubsectionvskip{0pt}

\settocsectionformat{\bfseries}
\settocsubsectionformat{\mdseries}
\settocsubsubsectionformat{\mdseries}
\setsectionpnumformat{\bfseries}
\setsubsectionpnumformat{\mdseries}
\setsubsubsectionpnumformat{\mdseries}

\let\oldtableofcontents\tableofcontents
\renewcommand{\tableofcontents}{%
  \vspace*{-\linespacing}% Default gap to top of CONTENTS is \linespacing.
  \oldtableofcontents}

\usepackage{mathrsfs,dsfont, comment,mathscinet}

\usepackage{enumerate,esint,accents}
\usepackage{mathtools}

\usepackage{graphicx}

\usepackage{epstopdf}
\usepackage[colorlinks=true, pdfstartview=FitV, linkcolor=blue, 
            citecolor=blue, urlcolor=blue]{hyperref}

\allowdisplaybreaks
\numberwithin{equation}{section}
\theoremstyle{plain}
\newtheorem{theorem}{Theorem}[section]
\newtheorem{proposition}[theorem]{Proposition}
\newtheorem{lemma}[theorem]{Lemma}

\theoremstyle{definition}
\newtheorem{definition}[theorem]{Definition}
\newtheorem{remark}[theorem]{Remark}

\usepackage{amssymb,amsmath}

\newcommand\R{\mathbb R}
\def\Rz {\mathbb{R}}
\newcommand\M{\mathbb M}

\newcommand\ep{\varepsilon}

\newcommand\wk{\rightharpoonup}

\newcommand{\C}{\mathbb{C}}

\newcommand\be[1]{\begin{equation}\label{#1}}
\newcommand\ee{\end{equation}}
\newcommand\ba[1]{\begin{align}\label{#1}}
\newcommand\ea{\end{align}}
\newcommand\bas{\begin{align*}}
\newcommand\eas{\end{align*}}
\newcommand\nn{\nonumber}
\newcommand\mtwo{\mathbb{M}^{2\times 2}}
\newcommand\HH{\mathcal{H}}
\newcommand\EEE{\color{black}}
\newcommand\UUU{\color{black}}

\newcommand\MMM{\color{black}}

\newcommand\Var{{\rm Var}}
\title[Young-Dupr\'e law for epitaxially-strained thin films]{Analytical validation of the Young-Dupr\'e law for epitaxially-strained thin films}
{}
\author[E. Davoli] {Elisa Davoli} 
\address[Elisa Davoli]{Department of Mathematics\\ University of Vienna\\Oskar-Morgenstern Platz 1\\1090 Vienna (Austria)}
\email[E. Davoli]{elisa.davoli@univie.ac.at}
\author[P.Piovano] {Paolo Piovano} 
\address[Paolo Piovano]{Department of Mathematics\\ University of Vienna\\Oskar-Morgenstern Platz 1\\1090 Vienna (Austria)}
\email[P. Piovano]{paolo.piovano@univie.ac.at}

\begin{document} 
\vskip .2truecm
\begin{abstract}
\small{\MMM We present here an analysis of the regularity of minimizers of a variational model for epitaxially strained thin-films identified by the authors in the companion paper \cite{davoli.piovano}. \EEE
The regularity of energetically-optimal film profiles is studied by extending previous methods and by developing new ideas based on transmission problems. The achieved regularity results relate to both the Stranski-Krastanow and the Volmer-Weber modes, the possibility of different elastic properties between the film and the substrate, and the presence of the surface tensions of all three involved interfaces: film/gas, substrate/gas, and film/substrate. 
Finally, geometrical conditions are provided for the optimal wetting angle, i.e., the angle formed at the contact point of films with the substrate. In particular, the Young-Dupr\'e law is shown to hold, yielding what appears to be the first analytical validation of such law for a thin-film model in the context of Continuum Mechanics.}
\end{abstract}
\maketitle

\tableofcontents
\section{Introduction}

Originally formulated in the context of Fluid Mechanics and sessile liquid drops \cite{DD,Y}, the \emph{Young-Dupr\'e law} characterizes the contact angle formed by drops at any touching point with their supporting surfaces (see Figure \ref{contactangle}). 
 As this condition involves both the  tension of supporting surfaces and the contact angles of drops (see Subsection \ref{subsec: FluidMechanics}), the law is often used to determine the unknown surface tension of certain materials by measuring the contact angles formed by different probe liquids. 

The use of this law is however not only restricted to liquid drops, but it has been naturally extended to \emph{epitaxy}, i.e., to the deposition of crystalline films on crystalline substrates \cite[Section 4.2.2]{Pohl}. Contact-angle conditions are in fact often essential for studying multiple-phase systems, as they represent the crucial boundary conditions for characterizing interface morphologies at (triple) \emph{junctions} \cite{BK,T}. A crucial difference between the setting of sessile drops and the one of thin-film deposition, though, is that in the latter elasticity has also to be taken into account as it might strongly affect the profile of the film. Indeed, the \emph{mismatch} between the crystalline lattices of the film and the substrate can induce large stresses in the film. In order to release such energy the atoms of the film move from their crystalline equilibrium to reach more favorable arrangements \cite{FG}.

Despite the applications of the Young-Dupr\'e law to elastic solids, a mathematical justification in the context of Continuum Mechanics seems to be missing in the Literature.
% Elasticity is crucial during thin-film deposition since the \emph{mismatch} between the crystalline lattices of the film and the substrate can induce large stresses in the film.  In order to release such energy film atoms move from their crystalline equilibrium to reach more favorable arrangements \cite{FG}. 
In this regard we refer the reader to \cite{SD} for a discussion on whether the presence of stresses modifies contact angles or not. In this paper we provide such mathematical validation in \emph{linear elasticity} in the context of thin films starting from the models introduced \MMM in \cite{davoli.piovano}. \EEE  Among our results we in particular find that the classical contact angles determined by the Young-Dupr\'e law are not impacted by the singular elastic fields present at the wedges of the contact corners. % stresses  Optimal contact angles are not impacted by elasticity.  

%eeds to be taken into account in epitaxy as is crucial during thin-film deposition especially when the material of the film and the ilm atoms move from their crystalline equilibrium to reach more favorable configurations the substrate are different as the lattice mismatch in fact can induce large stresses in the film and when a different material from the because of the presence of the \emph{lattice mismatch} between the material of the film and the substrate It is really because of the presence of the lattice mismatch that we are forced to include elasticity in the model. The lattice mismatch in fact can induce large stresses in the film, so that film atoms in order to release such energy move from the equilibrium to reach more favorable arrangements \cite{FG}. 

%Second that the authors analytically derive starting from the  \emph{transition-layer} and the \emph{sharp-interface} models introduced in \cite{S2}.

The purpose of this paper is therefore \MMM twofold. First, we show that optimal thin-film profiles of the variational heteroepitaxial thin-film model identified in \cite{davoli.piovano} (see \eqref{filmenergyfinal}) \EEE satisfy the Young-Dupr\'e law for angles $\theta\in[0,\pi/2]$ (see  Theorem \ref{thm:YDlaw}). \MMM Second, \EEE in Theorem \ref{thm:regularity} the regularity of the profile of minimizing configurations is assessed.

%\subsection{The objectives}  The purpose of this paper is threefold. First, in Theorem \ref{thm:justification} we analytically derive the variational model \eqref{filmenergyfinal} both as $\Gamma$-limit of the \emph{transition-layer models} introduced in \cite{S2}, and as relaxation of their associated \emph{sharp-interface model} (see Subsections \ref{subsec:sharp model} and \ref{subsec:transition model}). Second, we show that minimizing configurations of $\mathcal{F}$ satisfy geometrical properties which include the Young-Dupr\'e law for angles $\theta\in[0,\pi/2]$ (see  Theorem \ref{thm:YDlaw}). Third, we assess the regularity of minimizing configurations whose all regularity properties are listed in Theorem \ref{thm:regularity}. 

\subsection{The Young-Dupr\'e law in Fluid Mechanics}  
\label{subsec: FluidMechanics}  
The first formulation of the law dates back to 1805 and is due to Thomas Young \cite{Y} who derived it by computing the mechanical equilibrium of drops resting on planar surfaces under the action of the surface tensions $\gamma_f$, $\gamma_s$, $\gamma_{fs}$  of the three involved interfaces, respectively, the drop/gas interface, the substrate/gas interface, and the drop/substrate interface. Notice that we use here the subscript $f$ because in our setting drops coincide with film material.  Subsequently in 1869 a zero-angle condition for the case in which $\gamma_f\leq\gamma_s-\gamma_{fs}$ (also called \emph{wetting criterion} in \cite{Sp}) has been included in the relation by Anthanase Dupr\'e and Paul Dupr\'e (see \cite{DD}). A formulation of the law that includes both the contributions of \cite{Y} and \cite{DD} is 
  \begin{equation}\label{newyoung}
 \cos\theta=\frac{\min\{\gamma_f,\gamma_s-\gamma_{fs}\}}{\gamma_f},
\end{equation}
where  $\theta$ is the contact angle between planar surface of the substrate and the film profile (see Figure \ref{contactangle}). %, and  $\gamma_f>0$, $\gamma_s>0$, and $\gamma_s-\gamma_{fs}\geq 0$,  

  \begin{figure}[htp] 
\begin{center}
\includegraphics[scale=0.1]{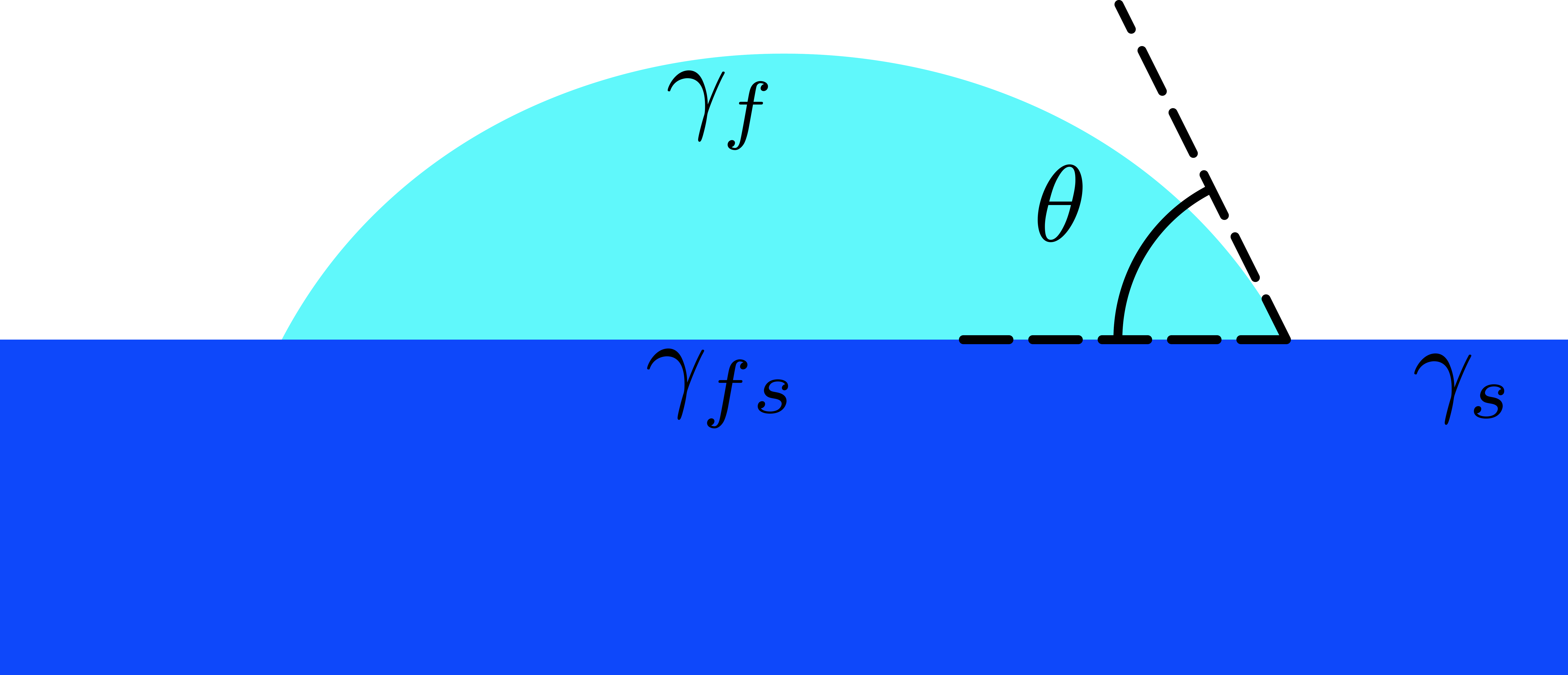}
\caption{Contact angle of a sessile drop.}
\label{contactangle}
\end{center}
\end{figure} 
In 1877 Carl Friedrich Gauss \cite{Gauss}  introduced a free energy consisting of four terms: a free surface energy related to the boundary of the drop detached from the substrate, a wetting energy accounting for the adhesion of the drop to the supporting surface and depending of an adhesion coefficient $\sigma$, a gravitational energy, and a Lagrange multiplier to include a  constraint on the volume of the drop. %Showing the existence and regularity for the minimizers of Gauss' energy is referred to classical  \emph{capillarity theory}. 
We recall that a law which includes this adhesion coefficient $\sigma$ has been formulated by Pierre Simon Laplace in \cite{Laplace} also starting from the ideas in \cite{Y}. This law which is often referred to as  \emph{Young-Laplace} law in the context of \emph{capillarity problems}, i.e., problems related to fluids in containers, can be stated as $\nu_D\cdot\nu_C=\sigma$, where $\nu_D$ and $\nu_C$ are the exterior normal to the drop and the container, respectively. We observe that \eqref{newyoung} is equivalent to the Young-Laplace law when $-\sigma$ corresponds to the right-hand side of \eqref{newyoung}.

However, also the results in the Literature related to the Young-Laplace law seem not to include elasticity. In particular, in \cite{CF} the authors prove that, if $\sigma\in(-1,0)$, than the detached boundary of the minimizing drops of the Gauss free energy is the graph of a function describing the thickness of the drop. This result, although in a different context, is in accordance with our analysis. In fact we assume that the admissible film profiles are \emph{graphs of height functions} and the conditions that we need to impose on $\gamma_f$, $\gamma_s$, $\gamma_{fs}$ are such that the right-hand side of \eqref{newyoung} belongs to $[0,1]$ (where, though, the boundary values can be included in our analysis). For more general conditions on the adhesion coefficient $\sigma$ we refer the reader to \cite{Baer} and \cite{DFM}, where every set of finite perimeter is an admissible drop, and the boundary regularity of optimal drops is studied also in the presence of anisotropy.

\subsection{The thin-film model} \label{subsec:model} 
The first rigorous validation of a thin-film \MMM energy \EEE as $\Gamma$-limit of the transition-layer model of \cite{S2} was performed in the seminal paper \cite{FFLM2}. \MMM In the companion paper \cite{davoli.piovano} we have provided a variational model taking \EEE into account the (possible) different elastic properties of the film and substrate materials, \MMM and thus particularly apt to describe \EEE \emph{heteroepitaxy}, i.e., the deposition of a material different from the one of the substrate. 

In order to describe \MMM the model identified in \cite{davoli.piovano}, \EEE we need to introduce some notation. \MMM As in the seminal paper \EEE \cite{S2} we \MMM model \EEE substrate and the film as continua, \MMM work in the setting \EEE of \emph{linear elasticity}, and \MMM consider \EEE two-dimensional profiles (or three-dimensional configurations with planar symmetry). \MMM We assume that interface between film and substrate is a subset of \EEE the $x$-axis, \MMM and we denote by the height function $h:[a,b]\to[0,+\infty)$ with $b>a>0$ the film thickness. The region occupied by the film and the substrate material is described by the subgraph of $h$, namely by the set \EEE $$\Omega_h:=\{(x,y):\,a<x<b,\,y<h(x)\},$$ \MMM whilst the film profile is encoded by the \emph{graph} \EEE
$$\Gamma_h:=\partial\Omega_h\cap\left((a,b)\times\R\right).$$
\MMM The \emph{material displacement} and its associated \emph{strain-tensor} are denoted by $u:\Omega_h\to \R^2$, and its symmetric part of the gradient, namely \EEE
$$Eu:={{\rm sym}}\nabla u,$$
\MMM respectively. As in \cite{FFLM2}, our description will include non-smooth profiles, and \EEE the height function \MMM will be \EEE assumed to be lower semicontinuous and with bounded pointwise variation. We \MMM will adopt the notation 
$${\tilde{\Gamma}_h}:=\partial\overline{\Omega}_h\cap\left((a,b)\times\R\right),$$
and $\Gamma_h^{cut}$ \MMM to identify \EEE the set of \emph{cuts} in the profile of $h$, namely $\Gamma_h^{cut}:={\Gamma}_h\setminus \tilde{\Gamma}_h$.

\MMM The lattice mismatch between the film and the substrate materials is known to induce large stresses and thus to play a major role in heteroepitaxy \cite{FG}. In our model the lattice mismatch is represented by means of a \EEE
parameter $e_0\geq 0$, \MMM and by the assumption that \EEE the minimum of the energy is \MMM attained \EEE at
$$E_0(y):=\begin{cases}e_0\, (\bf{e_1}\odot{\bf{e_1}})&\text{if }y\geq 0\\
0&\text{otherwise},
\end{cases}$$
where $({\bf e_1}, {\bf e_2})$ is the standard basis of $\R^2$. In the following we refer to $E_0$ as the \emph{mismatch strain}.  

The model considered in this paper, \MMM and rigorously validated in \cite{davoli.piovano}, is the \EEE energy functional $\mathcal{F}$, defined for any film configuration $(u,h)$ as
\begin{align}\label{filmenergyfinal}
 \mathcal{F}(u,h)\,=\,\int_{\Omega_h}W_{0}(y,&Eu(x,y)-E_0(y))\,dx\,dy\nonumber\\
&\qquad+\,\int_{\tilde{\Gamma}_h}\varphi(y)\,d\HH^1\,+\, \gamma_{\rm fs}(b-a)\,+\,2\gamma_{\rm f} \HH^1(\Gamma_h^{cut}),
\end{align}
 where the surface density $\varphi$ is given by
$$\varphi(y):=\begin{cases}\gamma_f&\text{if }y> 0,\\
 \min\{\gamma_f, \gamma_s-\gamma_{fs}\}&\text{otherwise,}\end{cases}$$
with 
\be{eq:ass-gammas}
\gamma_f>0,\quad \gamma_s>0,\quad\text{and}\quad\gamma_s-\gamma_{fs}\geq 0.
\ee
\MMM In the expression above, \EEE the elastic energy density \mbox{$W_0:\R\times \mtwo_{\rm sym}\to [0,+\infty)$} \MMM is defined as \EEE 
$$W_0(y,E):=\frac12 E:\C(y)E$$
for every $(y,E)\in \R\times \mtwo_{\rm sym}$,
\MMM where \EEE $\C(y)$ \MMM is \EEE the elasticity tensor, \MMM satisfying \EEE
$$\C(y):=\begin{cases} \C_{f}&\text{if }y>0,\\
\C_s&\text{otherwise},
\end{cases}$$
and \MMM such that \EEE
\be{eq:positive definite}E:\C(y)E>0\ee
for every $y\in\R$ and $E\in\mtwo_{\rm sym}$.
 The fourth-order tensors $\C_f$ and $\C_s$ are symmetric and positive-definite, and we allow them to be possibly different to include the case of a different elastic behavior  for  the film and the substrate. %Notice also that $\Gamma_{h,{\rm cut}}$ appearing in the last term of \eqref{filmenergyfinal} represent the jumps part in the graph of $h$. We are in fact in general assuming that 
 An alternative formulation of heteroepitaxy is to consider $E_0\equiv 0$ in \eqref{filmenergyfinal}, and to impose a \emph{transmission Dirichlet condition} at the interface between film and substrate. We refer to Remark \ref{rk:equivalent} to see that the two corresponding minimum problems are equivalent. \MMM As already highlighted in \cite{davoli.piovano}, \EEE energy functionals of the form \eqref{filmenergyfinal} represent the \emph{competition} between the \emph{roughening effect} of the elastic energy and the \emph{regularizing effect} of the surface energy that characterize the formation of such crystal microstructures (see \cite{FG,Ga,Gr} and \cite{FFLMi} for the related problem of crystal cavities) \MMM and are thus related to \EEE the study of \emph{Stress-Driven Rearrangement Instabilities  (SDRI)} \cite{Gr}. %Its minimization consists  in solving a \emph{free-boundary problem} as both the displacement $u$ and the set $\Gamma_h$ are unknowns of the problem. 
 \MMM We refer to \cite{davoli.piovano} and the references therein for an overview on the Literature on variational models in epitaxy and on related SDRI models. \EEE As already mentioned at the beginning of this subsection a similar functional to \eqref{filmenergyfinal} was derived  in \cite{FFLM2}  by $\Gamma$-convergence from the transition-layer model introduced in \cite{S2} in the case in which $\C_f=\C_s$, and $\gamma_{fs}=0$. We observe here that in \cite{FFLM2} the regularity of the local minimizers of such energy is studied for isotropic film and substrate in the case in which $\gamma_f\leq\gamma_s$, and the local minimizers  are shown to be smooth outside of finitely many \emph{cusps} and \emph{cuts} and to form zero contact angles with the substrate (see also \cite{BC,FFLMi}). We point out that the functional in \cite{FFLM2}, when restricted to the regime $\gamma_f\leq\gamma_s$ did not present any discontinuity along the film/substrate interface contained in the $x$-axis. The same applies for the energy in \cite{FM}. In our more general setting, instead, \eqref{filmenergyfinal}  always presents a sharp discontinuity with respect to the elastic tensors. Additionally the geometrical and regularity results of this paper include the \emph{dewetting regime}, $\gamma_f>\gamma_s-\gamma_{fs}$, for which the surface tension is also discontinuous. %of the film and the substrate along their interface contained in the $x$-axis.

\subsection{Wettability and growth modes} The importance of determining on which parameters contact angles in epitaxial growth depend, and of precisely characterizing their amplitude, resides on the need in applications to control the film adherence to substrates. The film adherence, that depends on the chemical interactions between the constituents of the two materials, is also referred to as  \emph{film wettability}. %fact that they provide a measure of the film \emph{wettability} with respect to a specific substrate, i.e., the ability of the film of adhering to a substrate as a result of chemical interactions between the atoms and the molecules of the two materials. 

Zero contact angles correspond to complete wetting that occurs when an infinitesimal thin layer of film atoms, the \emph{wetting layer}, spreads freely on the substrate and covers it. Positive angles instead represent the so called situation of nonspreading films, in which the substrate is partially exposed \cite{Zisman}. Since contact angles  represent the degree of the wettability of the film, they are in general also called  \emph{wetting angles}.

%extraordinary thin-film adaptability that consists of various possible
It is exactly because of their various possible morphologies and wettability properties that thin films play nowadays a key role in an ever-growing number of technologies which range from optoelectronics to semiconductor devices, and from solid oxide fuel/hydrolysis cells to photovoltaic devices. In fact, different modes of growth relate to different film wettability: \emph{Volmer-Weber} (VW) mode, in which separated islands form on top of the substrate, or situations in which the substrate is  completely covered such as in the \emph{Frank-van der Merwe} (FM) and \emph{Stranski-Krastanov} (SK) modes.  FM and SK differ as FM consists in a layer-by-layer  growth (next level starting only upon completion of previous layers), while SK presents islands which are nucleated on top of a wetting layer \cite{Pohl}.

%The extraordinary adaptability of thin films is due to the various possibilities of physical and chemical properties that they can display. Such properties strongly depend on the thin-film growth mode and can in principle be tuned by designing specific morphologies. 
Therefore, a large effort has been played at the engineering stage to improve the accuracy with which the resulting processed films correspond to the designed geometries.
%In the thin-film production a large effort is nowadays played to  the accuracy of thin-film profiles with the designed geometries since the physical and chemical properties of supported nanostructures  strongly depend on their morphology. 
Any advancement in the modeling that improves the engineering of pre-determined profile shapes has therefore a direct economical impact as it contributes to saving computational time needed for simulations, and to reducing the waste of material used in the current work-intensive and expensive trial-and-error production.
%The great variety of different growth possibilities is what makes thin films so widely used in an ever-growing number of applications, that ranges from optoelectronics to semiconductor devices, and from solid oxide fuel/hydrolysis cells to photovoltaic devices. 
We notice here that as a byproduct of our analysis, we also deduce that the VW thin-film mode (which corresponds to a positive wetting angle)  is exhibited if and only if $\gamma_f>\gamma_s-\gamma_{fs}$.

\subsection{Organization of the paper and methodology}

The paper is organized as follows.  In Section \ref{sec:mainresults} we introduce the mathematical setting and we rigorously state our main results (see Theorems \ref{thm:YDlaw}, and  \ref{thm:regularity}).

In  Section \ref{sec:regularity}, \MMM starting from the preliminary regularity results proved in \cite{davoli.piovano}, \EEE we develop a novel strategy for deriving contact-angle conditions. The originality of the method consists in implementing in our thin-film setting some ideas used for  \emph{transmission problems}, that rely on a \emph{decomposition formula} established in \cite{NS2}, as well as on the properties of the \emph{Mellin transform} and of the \emph{operator pencil}  (see \cite{NS}). In particular, by using the results in \cite{KS} we prove in Proposition \ref{thm:blow} a \emph{decay estimate} for the displacements corresponding to local minimizers of $\mathcal{F}$. 

In Section  \ref{sec:YD}, in view of Proposition \ref{thm:blow} we are able to perform a blow-up argument at the film/substrate contact points and to pass to the limit in the Euler equation satisfied by local minimizers (by considering variations only with respect to the profile functions). Among the contact points $Z_h$ of minimal profiles $h$ we distinguish the isolated ones  from the extrema of non-degenerate intervals in $Z_h$, and  we refer to the first as \emph{valleys} and to the latter as \emph{island borders}.  Careful choices of suitable competitors  for the minimal profile functions with respect to the different cases of valleys and island borders allow in Proposition \ref{pro:YDlawold} to identify corresponding contact-angle conditions. In particular the conditions proved in  Proposition \ref{pro:YDlawold} include the Young-Dupr\'e law for the \emph{wetting regime}, $\gamma_f\leq\gamma_s-\gamma_{fs}$. For the  \emph{dewetting regime}, $\gamma_f>\gamma_s-\gamma_{fs}$, the Young-Dupr\'e law  is obtained in Theorem \ref{thm:YDlaw} by a further comparison argument, that shows that angles smaller than the one characterized in \eqref{newyoung} are not energetically convenient in this regime. As a byproduct of our results we also obtain that in the dewetting regime there are no valleys, and hence, that islands are separated.

Finally, in Section \ref{sec:analiticity}  an adaptation of the proof strategy of Theorem \ref{thm:YDlaw},  together with improved decay estimates along the lines of \cite{FFLM2},  allow to reach in Theorem \ref{thm:regularity} the final regularity results for local minimizers. \newpage

\section{Main results}\label{sec:mainresults}

\subsection{Mathematical setting} \label{subset:mathset} \MMM We recall in this subsection the main definitions and the notation used throughout this paper and in \cite{davoli.piovano}. \EEE %, and we state the main results of the paper. 
We begin by characterizing the admissible film profiles. The set  $AP$  of admissible film profiles in $(a,b)$ is denoted by
$$AP(a,b):=\{h: [a,b]\to[0,+\infty)\,:\,\textrm{$h$ is lower semicontinuous and $\textrm{Var}\,h<+\infty$}\},$$
where $\textrm{Var}\,h$ denotes the pointwise variation of $h$, namely,
\begin{align*}
\textrm{Var}\,h :=\sup\Big\{  &\sum^n_{i=1}|h(x_i)-h(x_{i-1})|\,:\\
&\qquad \qquad\textrm{$P:=\{x_1,\dots,x_n\}$ is a partition of $[a,b]$} \Big\}.
\end{align*}
We recall that for every lower semicontinuous function $h: [a,b]\to[0,+\infty)$, to have finite pointwise variation is equivalent to the condition 
$$\HH^1(\Gamma_h)<+\infty,$$
where
$$\Gamma_h:=\partial\Omega_h\cap\left((a,b)\times\R\right).$$
 For every  $h\in AP(a,b)$, and for every $x\in (a,b)$, consider the left and right limits
$$ h(x^\pm):=\lim_{z\to x^{\pm}} h(z),$$
we define
$$h^-(x):=\min\{h(x^+),h(x^-)\}=\liminf_{z\to x} h(z),$$
and 
$$h^+(x):=\max\{h(x^+),h(x^-)\}=\limsup_{z\to x} h(z).$$
In the following  ${\rm Int}(A)$ denotes the interior part of a set $A$.
Let us now recall some properties of height functions $h\in AP(a,b)$, regarding their graphs $\Gamma_h$, their subgraphs $\Omega_h$, the film and the substrate parts of the subgraph,
$$\Omega_h^+:=\Omega_h\cap\{y>0\}$$
and
$$ \Omega_h^-:=\Omega_h\cap\{y\leq 0\}$$ respectively, and the sets 
\be{eq:Gamma-tilda}\tilde{\Gamma}_h:=\partial\bar{\Omega}_h\cap ((a,b)\times \R).\ee
%\begin{remark}
Any $h\in AP(a,b)$ satisfies the following assertions (see \cite[Lemma 2.1]{FFLM2}):
\begin{enumerate}
\item[1.] $\Omega_h^+$ has finite perimeter in $((a,b)\times \R)$,
\item[2.]  $\Gamma_h=\{(x,y):\,a<x<b,\,h(x)<y<h^+(x)\}$,
\item[3.]  $h^-$ is lower semicontinuous and ${\rm Int}\left(\overline{\Omega}\right)=\{(x,y):\,a<x<b,\, y<h^-(x)\}$,
\item[4.]  $\tilde{\Gamma}_h=\{(x,y):\,a<x<b,\,h^-(x)\leq y\leq h^+(x)\}$,
\item[5.]  $\Gamma_h$ and $\tilde{\Gamma}_h$ are connected.
\end{enumerate}
%\end{remark}

We now characterize various portions of $\Gamma_h$. To this aim we denote the \emph{jump} set of a function $h\in AP(a,b)$, i.e., the set of its profile discontinuities, by
\be{eq:def-S-h}J(h):=\{x\in (a,b):\,h^-(x)\neq h^+(x)\},\ee
whereas the set identifying \emph{vertical cuts} in the graph of $h$ is given by
\be{eq:def-S}C(h):=\{x\in (a,b):\,h(x)<h^-(x)\}.\ee
The graph $\Gamma_h$ of a height function $h$ is then characterized by the decomposition 
$$\Gamma_h=\Gamma_h^{jump}\sqcup\Gamma_h^{cut}\sqcup\Gamma_h^{graph},$$
where $\sqcup$ denotes the disjoint union, and 
\begin{align}
\nn\Gamma_h^{jump}&:=\overline{\{(x,y):\,x\in (a,b)\cap J(h),\, h^-(x)\leq y\leq h^+(x)\}},\\
\Gamma_h^{cut}&:=\{(x,y):\,x\in (a,b)\cap C(h),\, h(x)\leq y< h^-(x)\},\label{eq:cuts}\\
\nn\Gamma_h^{graph}&:=\Gamma_h\setminus(\Gamma_h^{jump}\cup\Gamma_h^{cut}). 
\end{align}
%and 
%$$\Gamma_h^{graph}:=\Gamma_h\setminus(\Gamma_h^{jump}\cup\Gamma_h^{cut}).$$

 \begin{figure}[htp]
\begin{center}
\includegraphics[width=12.5cm]{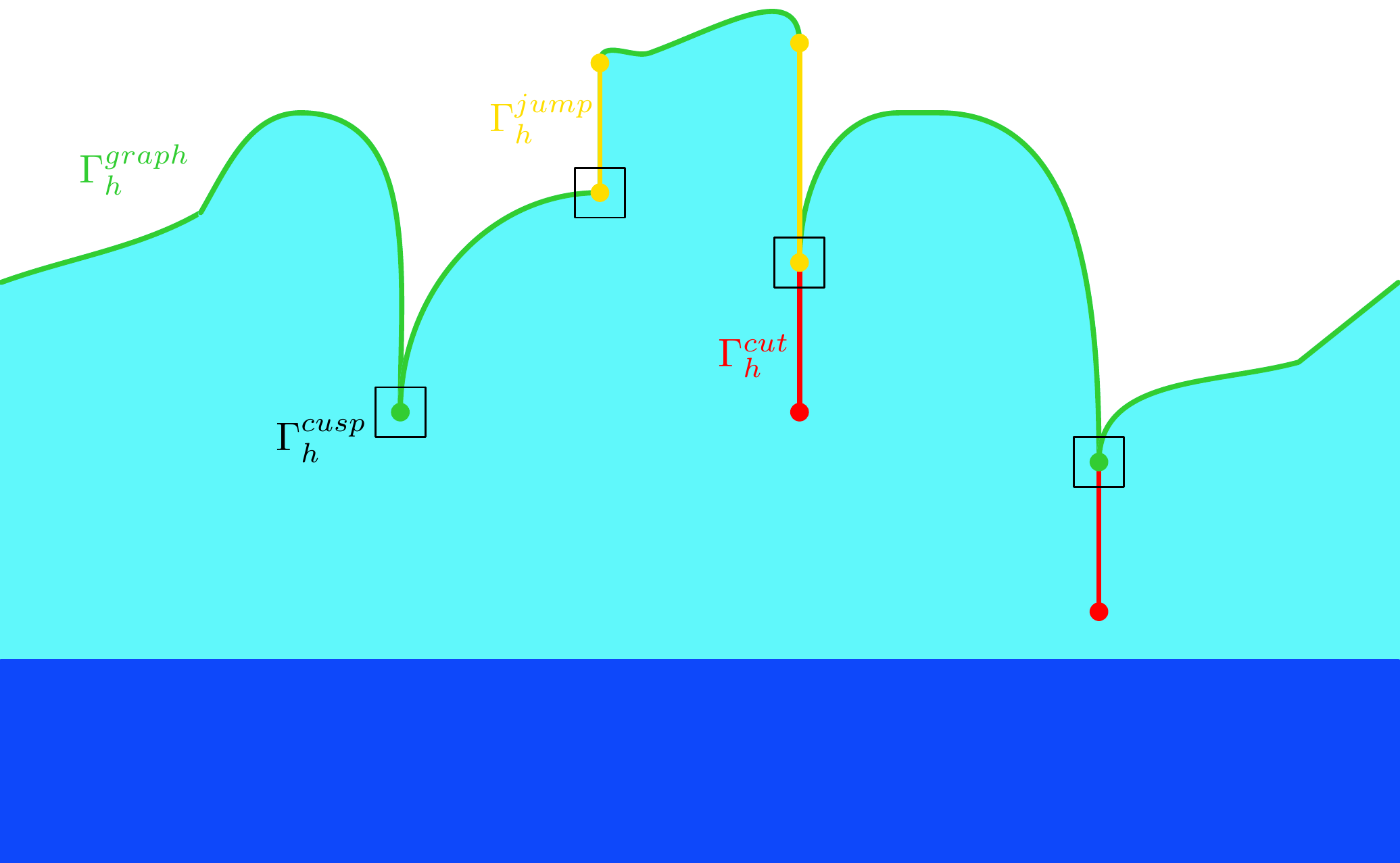}
\caption{In the figure above an admissible profile function $h$ is displayed. The portions of $\Gamma_h$ corresponding to $\Gamma_h^{graph}$, $\Gamma_h^{jump}$,  and $\Gamma_h^{cut}$  are represented with the colors green, yellow, and red respectively. The points in $\Gamma_h^{cusp}$ are marked by enclosing them within squares.}
\label{graphs}
\end{center}
\end{figure}
We observe that  $\Gamma_h^{graph}$ represents the regular part of the graph of $h$, whilst both $\Gamma_h^{jump}$ and $\Gamma_h^{cut}$ consist in (at most countable) unions of segments, corresponding to the \emph{jumps} and the \emph{cuts} in the graph of $h$, respectively (see Figure \ref{graphs}). Notice also that
$$\Gamma_h=\tilde{\Gamma}_h\sqcup \Gamma_h^{cut}.$$

Let us also identify  the set of \emph{cusps} in $\Gamma_h$ by
\begin{align*}
\Gamma_h^{cusp}:=\big\{ (x,h^-(x))\,:\,\,&\textrm{either $x\in J(h)$}\\
 &\textrm{or we have that $x\not\in J(h)$ with $h'_+(x)=+\infty$ or $h'_-(x)=-\infty$}\big\}
\end{align*}
(see Figure \ref{graphs}). % and denote by $\Gamma_h^{reg}$ the set 
%\be{eq:gamma-h-reg}
%\Gamma_h^{reg}:=\Gamma_h\setminus (\Gamma_h^{cut}\cup \Gamma_h^{cusp})=\tilde{\Gamma}_h\setminus \Gamma_h^{cusp}.
%\ee

For every $h\in AP(a,b)$ we indicate its set of of zeros by
$$Z_h:=\Gamma_h\cap\{x\in[a,b]\,:\,h(x)=0\}.$$
For every $x\in Z_h$, let $\theta^{\pm}(x)$ be the internal angles, with amplitude smaller or equal to $\frac{\pi}{2}$, between the $x$-axis and the tangents to $\Gamma_h$ in $(x,0)$ from the left and from the right, with slopes $h'_{-}(x)$ and $h'_{+}(x)$, respectively.
Consider the set 
$$I_h:=\{(c,d)\subset Z_h:\,\textrm{$c<d$ and $c,d\not\in {\rm Int}(Z_h)$}\},$$
and let
$$P_h:=Z_h\setminus \bigcup_{(c,d)\in I_h}[c,d].$$
\begin{figure}[htp]
\begin{center}
\includegraphics[width=11 cm]{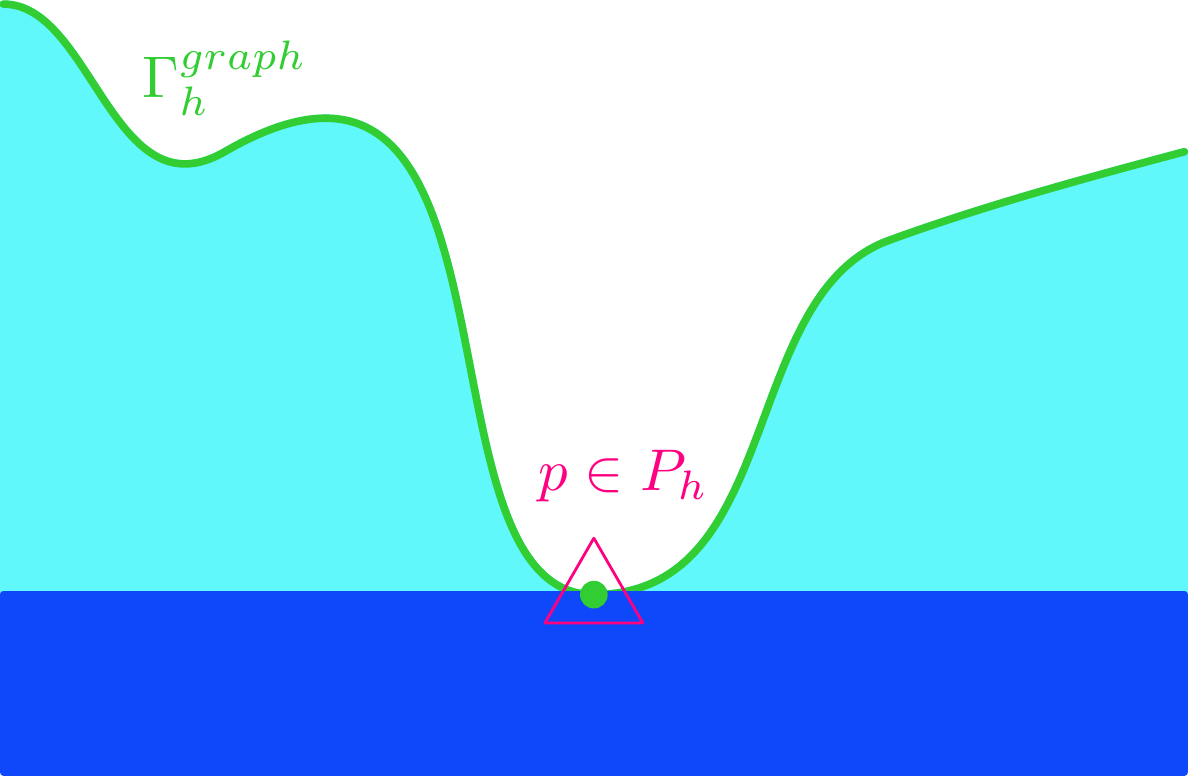}
\caption{A valley at an isolated point $p\in P_h$ is displayed. The point $p$ is indicated by enclosing it in a pink triangle}
\label{ph}
\end{center}
\end{figure}
\begin{figure}[htp]
\begin{center}
\includegraphics[width=11 cm]{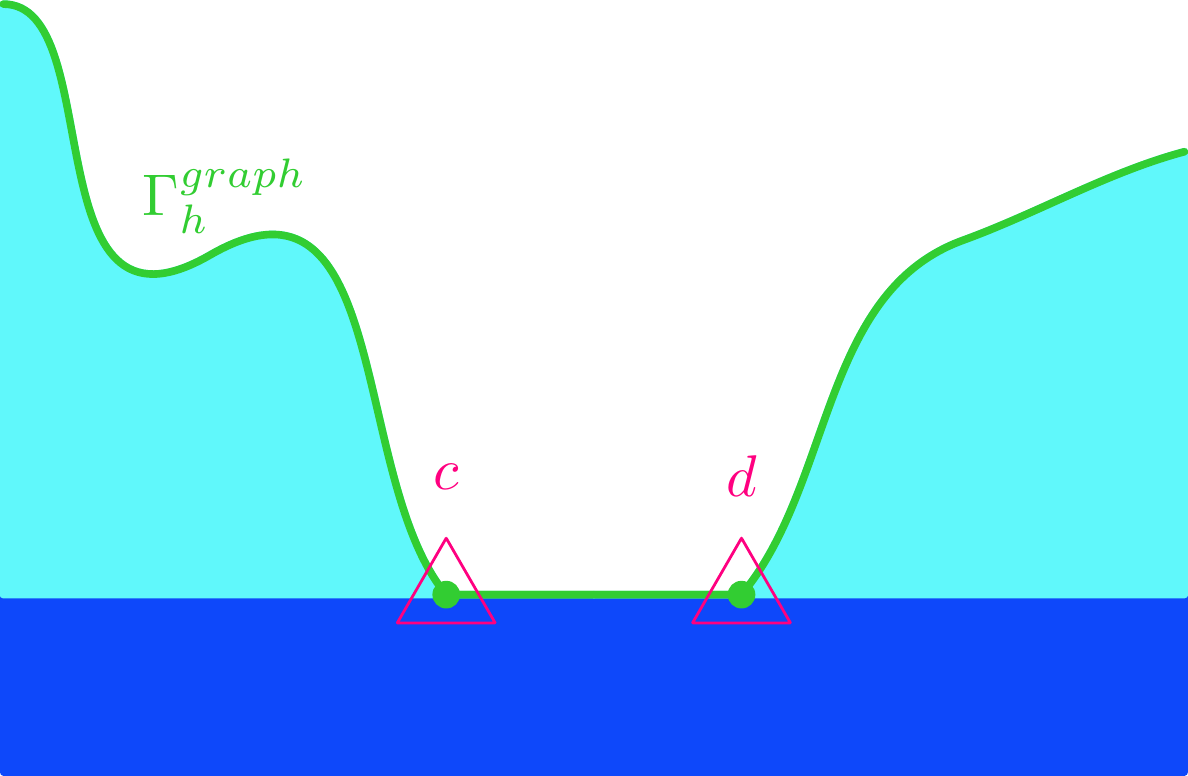}
\caption{An interval $(c,d)\in I_h$ is displayed. The points $c,d$ indicated with pink triangles are the only ones in $I_h$ with non-trivial contact angles.}
\label{ih}
\end{center}
\end{figure}

We will refer to the endpoints $c$ and $d$ of any interval $(c,d)\in I_h$ as \emph{borders of} (\emph{two different}) \emph{islands} and to the points in $P_h$ as \emph{valleys}, and we observe that 
\be{eq:s8star}
\begin{cases} 
\theta^{-}(x)=0&\text{for every }x\in (c,d]\\
\theta^{+}(x)=0&\text{for every }x\in [c,d)
\end{cases}
\ee
(see Figures \ref{ph} and \ref{ih}).

We now define the family $X$ of admissible film configurations as
$$X:=\{(u,h):\,\textrm{$u\in H^1_{\rm loc}(\Omega_h;\R^2)$ and $h\in AP(a,b)$}\}$$ 
and we endow $X$ with the following notion of convergence.

\begin{definition}
\label{def:conv-X}
We say that a sequence $\{(u_n,h_n)\}\subset X$ converges to $(u,h)\in X$, and we write $(u_n,h_n)\to (u,h)$ in $X$ if
\begin{enumerate}
\item[1.] $\sup_n \Var\, h_n<+\infty$,
\item[2.] $\R^2\setminus\Omega_{h_n}$ converges to $\R^2\setminus\Omega_h$ in the Hausdorff metric,
\item[3.] $u_n\wk u$ weakly in $H^1(\Omega';\R^2)$ for every $\Omega'\subset\subset\Omega_h$.
\end{enumerate}
\end{definition}
Let us also consider the following subfamily $X_{\rm Lip}$ in $X$ of configurations with Lipschitz profiles, namely,
$$X_{\rm Lip}:=\{(u,h):\,u\in H^1_{\rm loc}(\Omega_h;\R^2),\,h\text{ is Lipschitz}\}.$$
We recall from Subsection \ref{subsec:model} that the thin-film model analyzed in this paper is characterized by the energy $\mathcal{F}$ defined by \eqref{filmenergyfinal} on configurations $(u,h)\in X$.

We state here the definition of \emph{$\mu$-local minimizers} of the energy $\mathcal{F}$. 
\begin{definition}
\label{def:local-min}
We say that a pair $(u,h)\in X$ is a $\mu$-local minimizer of the functional $\mathcal{F}$ if $\mathcal{F}(u,h)<+\infty$ and there exists $\mu>0$ such that
$$\mathcal{F}(u,h)\leq \mathcal{F}(v, g)$$
for every $(v,g)\in X$ satisfying $|\Omega_g^+|=|\Omega_h^+|$ and $|\Omega_g\Delta \Omega_h|\leq \mu$.
\end{definition}
\noindent Note that every global minimizer (with or without volume constraint) is a $\mu$-local minimizer.

\subsection{Statement of the main results} The paper contains \MMM two \EEE main theorems. \MMM Consider \EEE the situation in which $\C_f$ and $\C_s$ are the elasticity tensors of isotropic materials with Lam\'e coefficients $\mu_f$, $\lambda_f$, and $\mu_s$, $\lambda_s$, respectively.

Our \MMM first \EEE result regards the identification of contact angle conditions for the $\mu$-local minimizers $(u,h)\in X$ of  $\mathcal{F}$.

\begin{theorem}[Contact-angle conditions]
\label{thm:YDlaw} Assume that the  Lam\'e coefficients of the film and the substrate satisfy %\emph{quasi-monotonicity} condition:
\be{monotonicity}\mu_s\geq \mu_f>0\quad\text{and}\quad\mu_s+\lambda_s\geq\mu_f+\lambda_f>0.
\ee
%Let the film and the substrate be isotropic materials satisfying the \emph{quasi-monotonicity} condition:
%\be{monotonicity}\mu_s\geq \mu_f\quad\text{and}\quad\mu_s+\lambda_s\geq\mu_f+\lambda_f.
%\ee
%where  $\mu_f$, $\lambda_f$, and $\mu_s$, $\lambda_s$ are their corresponding Lam\'e coefficients.

Then, every $\mu$-local minimizer $(u,h)\in X$ of $\mathcal{F}$ satisfies the following properties:
\begin{enumerate}
%\item[1.] If $h(x_0)>0$, then $\omega(z_0)=\pi$.
\item[1.] For every $p,c,d\in Z_h\setminus(\Gamma_h^{cusp}\cup\Gamma_h^{cut})$ such that $p\in P_h$ and $(c,d)\in I_h$ we have
$$\theta^{-}(p)=\theta^{+}(p)=\theta^{-}(c)=\theta^+(d)=\arccos(\beta),$$
where 
\be{eq:beta}
\beta:=\frac{\min\{\gamma_f,\gamma_s-\gamma_{fs}\}}{\gamma_f}.
\ee
\item[2.] If $\beta<1$, then $P_h\setminus(\Gamma_h^{cusp}\cup\Gamma_h^{cut})=\emptyset$.
%\item[1.] If $h(x_0)>0$, then $\omega(z_0)=\pi$.
\item[3.]  If $\beta\neq 0$, then  $\Gamma_h^{jump}\cap Z_h= \emptyset$.
%\item[4.] if $\beta=0$, then every connected component of $\Gamma_h^{jump}$ intersects the $x$-axis.
%\item[5.] if $\beta<1$, then $P_h=\emptyset$.
\end{enumerate}
\end{theorem}

%Notice that in particular asserts 2. and 3. of Theorem \ref{thm:YDlaw} implies also that $Z_h\cap J(h)=\emptyset$.

We remark that  Theorem \ref{thm:YDlaw} is the analytical validation of  the Young-Dupr\'e law for angles not greater than $\pi/2$. Let us sum up here the possible scenarios for the wetting angles:
\begin{description}
\item[Wetting regime] For $\gamma_s-\gamma_{fs}\geq\gamma_f$  all contact angles of $\Gamma_h\setminus (\Gamma_h^{cut}\cup \Gamma_h^{cusp})$ are zero.
\item[Dewetting regime] For $\gamma_s-\gamma_{fs}<\gamma_f$ all nontrivial contact angles $\theta$ of points in $Z_h\setminus (\Gamma_h^{cut}\cup \Gamma_h^{cusp})$ are such that 
$$\cos\theta=\frac{\gamma_s-\gamma_{fs}}{\gamma_f}.$$
\end{description}
We stress that, in agreement with the Young-Dupr\'e law, jumps at island borders (see Figure \ref{contactangles1}) are excluded when $\min\{\gamma_f,\gamma_s-\gamma_{fs}\}/\gamma_f\not=0$.  
\begin{figure}[htp]
\begin{center}
\includegraphics[scale=0.8]{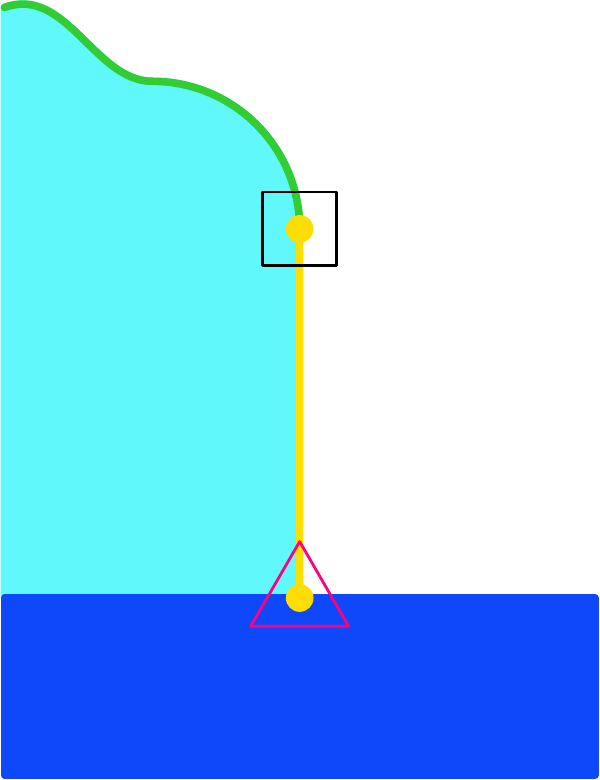}
%\qquad
%\includegraphics[scale=0.24]{contact2}\\
\caption{Example of a jump at an island border (here indicated with a pink triangle). This is the only type of jump allowed by Theorem \ref{thm:YDlaw} and only if $\min\{\gamma_f,\gamma_s-\gamma_{fs}\}/\gamma_f=0$.}
\label{contactangles1}
\end{center}
\end{figure}
Note also that the contact angles at valleys are always zero (and there are no jumps at valleys), since valleys exist only for the wetting regime when $\beta=1$.

However, our analysis allows the set $D_h:=\left[\left(\Gamma_h^{cusp}\cup\Gamma_h^{cut}\right)\setminus\Gamma_h^{jump} \right]\cap Z_h$ to be nonempty. It seems though that this is not a restriction of our method but it is  in agreement with the experimental evidence. Points in $D_h$ may represent in fact \emph{dislocations} that are experimentally shown to form as a further mode of strain relief and to migrate at the film/substrate interface. We kindly refer the reader to \cite{FFLM4} and the reference therein for more details on dislocations in epitaxy and for a thin-film model accounting for their presence. Some examples of contact angles in $D_h$ are displayed in Figure \ref{contactangles2}.
\begin{figure}[htp]
\begin{center}
\includegraphics[width=5.7cm]{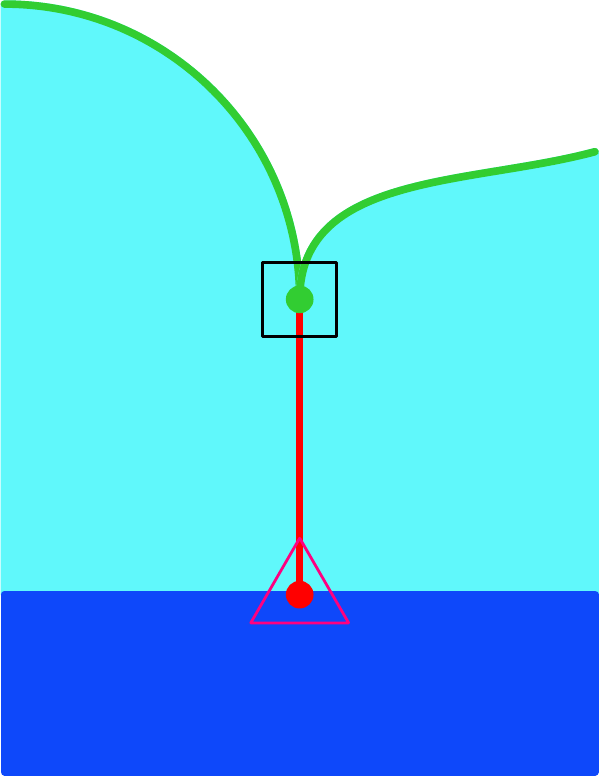}
\qquad
\includegraphics[width=5.7cm]{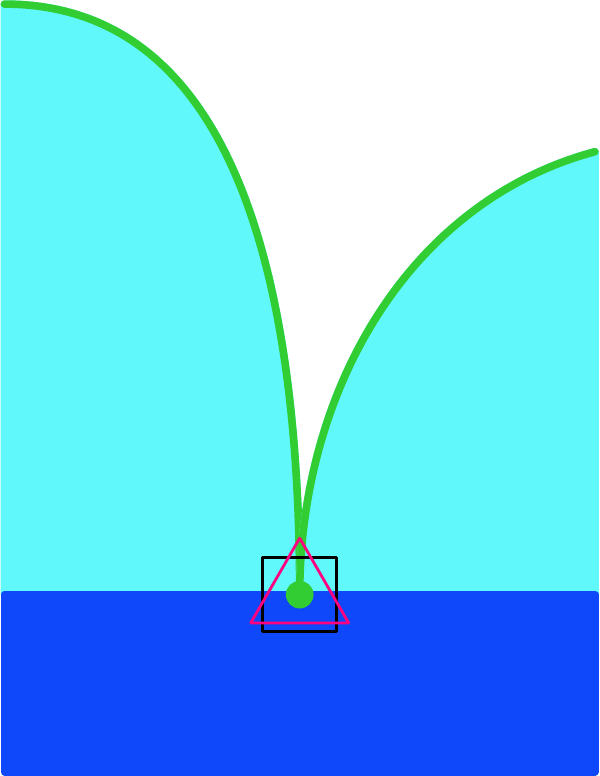}\\
\caption{Cuts (left)  and cusps (right) may represents dislocations at the film/substrate interface.}
\label{contactangles2}
\end{center}
\end{figure}

Regarding condition \eqref{monotonicity}, assuming $\mu_f,\,\mu_s>0,\,\lambda_f+\mu_f>0$, and $\lambda_s+\mu_s>0$ guarantees the ellipticity of the transmission problem associated to the Euler-Lagrange equations of $\mu$-local minimizers of $\mathcal{F}$ (see \cite[Lemma 1.3]{Knees-DiplomaThesis}). The assumption
\be{onlymonoticity}\mu_s\geq \mu_f\quad\text{and}\quad\mu_s+\lambda_s\geq\mu_f+\lambda_f
\ee
 is a \emph{quasi-monotonicity condition}. This kind of assumptions are classically considered in transmission problems for elliptic systems, we refer the reader  to  \cite{DSW} for the first formulation for transmission problems with the Laplace operator (see also \cite{Knees2002} and the references therein). As stated in \cite{KS} where \eqref{onlymonoticity} is introduced, ``it seems that the quasi-monotonicity condition $[ \dots]$ %does not only lead to higher regularity results but also 
describes a class of composites which can sustain higher loads before breaking''. Furthermore, condition \eqref{onlymonoticity} implies that the \emph{shear} and the \emph{P-wave moduli} of the substrate are higher than those of the film. As such parameters are \emph{elastic moduli} for the materials, this entails that the substrate is stiffer than the film. Such requirement appears to be natural in the thin-film models here considered from \cite{S2}, where only the boundary of the film and not the boundary of the substrate is allowed to deform. We recall that in these models the film/substrate interface is forced to coincide with the $x$-axis. 
As a matter of fact, quasi-monotonicity conditions are strongly related to the particular geometry in which the transmission problem is considered, and in particular to the position of the transmission interface at boundary corners. Other conditions than \eqref{onlymonoticity} might be included if the film/substrate interface is not maintained fixed as in \cite{S2}. %by forcing it to coincide with the $x$-axis. 
%Furthermore, the condition seems to be related to the fact that the thin- film models considered allow only the boundary of the film and not the boundary of the substrate to deform. In fact, they imply that the substrate is stiffer than the film as 

The final main theorem of the paper concerning the regularity of optimal profiles is the following.

\begin{theorem}[Regularity]
\label{thm:regularity} Assume that the  Lam\'e coefficients of the film and the substrate satisfy  \eqref{monotonicity}.%\emph{quasi-monotonicity} condition:

%Assume that the Lam\'e coefficients satisfy the monotonicity condition \eqref{monotonicity}.
Then, every $\mu$-local minimizer $(u, h)\in X$ of $\mathcal{F}$ has the following regularity properties:
\begin{enumerate}
\item[1.] Cusps points and vertical cuts are at most finite;
%\item[2.] $Z$ is the finite union of (possibly degenerate) closed intervals;
\item[2.] $\Gamma_h^{reg}:=\Gamma_h\setminus (\Gamma_h^{cut}\cup \Gamma_h^{cusp})$
 %downward cusps
is locally the graph of a Lipschitz function;
%\item[3.]  \PPP $\Gamma_h^{reg}$ is $C^1$ at every $(x,h^-(x))\in \Gamma_h^{jump}\setminus Z_h$; \BBB
\item[3.]  $\Gamma_h^{reg}\setminus Y_h$ is $C^{1,\alpha}$ for all $\alpha\in(0,1/2)$, where $Y_h$ is the subset of $Z_h\cap \Gamma_h^{reg}$ containing points with nonzero contact angles for $h$;
\item[4.] The set 
$$A_h:=\begin{cases}\Gamma_h^{reg}\setminus Z_h&\text{if }\C_f\not=\C_s\\
\Gamma_h^{reg}\setminus Y_h&\text{if }\C_f=\C_s
\end{cases}$$
 is analytic and satisfies the Euler-Lagrange equation
\be{eq:el-analyt-h}
\gamma_f k_{A_h}=\tau_{A_h}\,(W_0(\cdot, Eu(\cdot)- E_0))+\lambda_0\quad\text{on }A_h,
\ee
where the function $k_{A_h}(\cdot)$ denotes the curvature of  $A_h$, $\tau_{A_h}(\cdot)$ is the trace operator on $A_h$, and $\lambda_0$ is a suitable Lagrange multiplier.

\end{enumerate}
\end{theorem}

We also point out that for $\C_f=\C_s$ in view of Assertion 5. of Theorem \ref{thm:regularity} for every $\mu$-local minimizer $(u,h)$ of $\mathcal{F}$ the set $Z_h$ has either finite cardinality or nonempty interior in the $x$-axis. Finally, we observe that in the wetting regime for $\C_f=\C_s$ the analytic portion of the graph $A_h$ coincides with $\Gamma_h^{reg}$ since by the assertions 1. and 2. of Theorem \ref{thm:YDlaw} we have $Y_h=\emptyset$.
\begin{remark}
\label{rk:equivalent}
The results in Theorem \ref{thm:YDlaw} hold also for $\mu$-local minimizers of the energy
\begin{align}\label{eq:alternative}
\mathcal{E}(u^+,u^-,h)&:=\int_{\Omega_h^+}\C_f Eu^+(x,y):Eu^+(x,y)\,dx\,dy\\
&\nonumber\qquad+\int_{\Omega_h^-}\C_s Eu^-(x,y):Eu^-(x,y)\,dx\,dy+\int_{\Gamma_h}\varphi(y)\,d\HH^1\\
&\nonumber\qquad+\, \gamma_{fs}(b-a)+2\gamma_f\HH^1(\Gamma_h^{cut})
\end{align}
for every $(u^+,u^-,h)\in \tilde{X}$, where 
\begin{align*}
\nonumber \tilde{X}:=&\{(u^+,u^-,h):\,\textrm{$u^+\in H^1_{\rm loc}(\Omega_h^+;\R^2)$,\,$u^-\in H^1_{\rm loc}(\Omega_h^-;\R^2)$,}\\
&\qquad\qquad\,\textrm{$u^+(\cdot,0)-u^-(\cdot,0)=(e_0 \cdot,0)$, and $h\in AP(a,b)$}\}.
\end{align*}
In fact, there is a 1-1 correspondence between triples $(u^+,u^-,h)$ that are $\mu$-local minimizers of \eqref{eq:alternative}, and pairs $(u,h)$ which are $\mu$-local minimizers of \eqref{filmenergyfinal}, with
$$
u(x,y):=\begin{cases} u^+(x,y) - (e_0x,0) &\textrm{if $y\geq 0$}\\
u^-(x,y) &\textrm{if $y<0$}
\end{cases}
$$
for $(x,y)\in\Omega_h$. Energy functionals similar to \eqref{eq:alternative} are considered for the corresponding evolution problem (see, e.g., \cite{TS}).

\end{remark}

%\section{Relaxation and $\Gamma$-convergence}

\section{Properties of local minimizers}\label{sec:regularity}

In this section we start analyzing the regularity of $\mu$-local minimizers $(u,h)$ of \eqref{filmenergyfinal}. In the first subsection we \MMM recall the results in \cite{davoli.piovano}, showing \EEE that optimal profiles $h$ satisfy the \emph{internal-ball condition}.  The second subsection is devoted to establish a decay estimate for the minimizing displacements $u$, and relies on some techniques introduced in the setting of transmission problems for elliptic systems (see \cite{Knees-DiplomaThesis, Knees2002, NS,NS2}).

\subsection{Internal-ball condition}\label{subsec:ball condition} 
%The first part of our analysis is very close to the arguments in \cite[Section 3]{FFLM2}, therefore we only state the main results without proof.  
\MMM In this subsection we collect some first regularity results for local minimizers. We refer to \cite{davoli.piovano} for the proofs of the next two propositions. The first observation is \EEE that the area constraint in the  minimization problem of Definition \ref{def:local-min} can be replaced with a suitable penalization in the energy functional. 

\begin{proposition}\label{prop:penalization}
Let $(u, h)\in X$ be a $\mu$-local minimizer for the functional $\mathcal{F}$. Then there exists $\lambda_0>0$ such that
\be{eq:penprob}
\mathcal{F}(u, h)=\min\left\{\mathcal{F}(v, g)+\lambda||\Omega_h^+|-|\Omega_g^+||:\,(v,g)\in X,\,|\Omega_g\Delta \Omega_h|\leq \frac{\mu}{2}\right\}
\ee
for all $\lambda\geq \lambda_0$.
\end{proposition}

We are now ready to \MMM recall \EEE the internal-ball condition for optimal profiles.

%The analogue to \cite[Proposition 3.3]{FFLM2} holds in our framework.
\begin{proposition}[Internal-ball condition]
\label{thm:palla-interna}
Let $(u,h)\in X$ be a $\mu$-local minimizer for the functional $\mathcal{F}$. Then, there exists $\rho_0>0$ such that for every $z\in \overline{\Gamma}_h$ we can choose a point $P_z$ for which  $B(P_z,\rho_0)\cap((a,b)\times \R)\subset \Omega_h$, and 
$$\partial B(P_z,\rho_0)\cap\overline{\Gamma}_h=\{z\}.$$
\end{proposition}

\noindent We point out that in view of Proposition \ref{thm:palla-interna}  the upper-end point of each \emph{cut} is a cusp point (see Figure \ref{graphs}). 

The following proposition is a consequence of the internal-ball condition.

\begin{proposition}
\label{prop:loc-lip}
Let $(u,h)\in X$ be a $\mu$-local minimizer for the functional $\mathcal{F}$. Then for any $z_0\in \overline{\Gamma}_h$ there exist an orthonormal basis $\mathbf{v}_1,\mathbf{v}_2\in\R^2$, and a rectangle $$Q:=\left\{z_0+s\mathbf{v}_1+t\mathbf{v}_2:\,-a'<s<a',\,-b'<t<b'\right\},$$
$a',b'>0$, such that $\Omega_h\cap Q$ has one of the following two representations:
\begin{itemize}
\item[1.] There exists a Lipschitz function $g:(-a',a')\to(-b',b')$ such that $g(0)=0$ and $$\qquad\Omega_h\cap Q:=\left\{z_0+s\mathbf{v}_1+t\mathbf{v}_2:\,-a'<s<a',\,-b'<t<g(s)\right\}\cap ((a,b)\times \R).$$
In addition, the function $g$ admits left and right derivatives at all points that are, respectively, left and right continuous.
\item[2.] There exist two Lipschitz functions $g_1,g_2:[0,a')\to(-b',b')$ such that $g_i(0)=(g_i)'_+(0)=0$ for $i=1,2$, $g_1\leq g_2$, and
               $$\quad\qquad\Omega_h\cap Q:=\left\{z_0+s\mathbf{v}_1+t\mathbf{v}_2:\,0<s<a',\,-b'<t<g_1(s)\text{ or }g_2(s)<t<b'\right\}.$$
In addition, the functions $g_1,g_2$ admit left and right derivatives at all points that are, respectively, left and right continuous.
\end{itemize}
\end{proposition}

For the proof of Proposition \ref{prop:loc-lip} we refer the reader to  \cite[Lemma 3]{CL} and \cite[Proposition 3.5]{FFLM2}. In particular Proposition \ref{prop:loc-lip} entails that the set 
$$\Gamma_h^{reg}=\Gamma_h\setminus(\Gamma_h^{cusp}\cup\Gamma_h^{cut})$$
is locally Lipschitz (see the proof of Theorem \ref{thm:regularity} in Section \ref{sec:analiticity} for more details). 

\subsection{Decay estimate}\label{sec:decay} 
From now on we work under the assumption that both the film and the substrate are made of linearly elastic isotropic materials, and we denote by $\mu_f$, $\lambda_f$, $\mu_s$, $\lambda_s$ their Lam\'e coefficients. Note that
$$\C_{\sigma}Eu=2\mu_{\sigma}Eu+\lambda_{\sigma}({\rm div}\,u)Id,\quad\sigma=f,s,$$
for every $u\in H^1(\Omega_h;\R^2)$.

In order to prove the decay estimate of Proposition \ref{thm:blow} for minimizing configurations $(u,h)$ at the points of $\Gamma_h^{reg}$ a \emph{blow-up} around such points  is needed. As the graph is allowed to touch the film/substrate interface, we are lead to consider transmission problems for Lam\'e systems in conical sets. We first state a preliminary lemma, relying on \cite[Theorem 1.5.2.8]{Grisvard}, and whose proof is contained in \cite[Lemma 3.12]{FFLM2}).

\begin{lemma}
\label{lemma:grisvard}
Let $\mathcal{C}$ be a circular sector of amplitude $\theta\in (0,2\pi)$ and radius $R>0$. Assume that $\mathcal{C}$ is the reference configuration of a linearly elastic isotropic material whose Lam\'e coefficients are denoted by $\mu$ and $\lambda$. Let $g\in H^{1/2}(\partial \mathcal{C};\R^2)$ be a function vanishing in a neighborhood of the origin. Then there exists a function $v\in H^2(\mathcal{C};\R^2)$ such that
$$[2\mu Ev+\lambda ({\rm div }\,v)Id\,]\,\nu_{\mathcal{C}}=g\quad\text{on }\partial \mathcal{C},$$
where $\nu_{\mathcal{C}}$ is the outer unit normal to $ \mathcal{C}$ (where it exists), and
$$v=0\quad\text{on }\partial\mathcal{C}.$$
\end{lemma}

In the following proposition we assess the regularity of weak solutions to transmission problems for Lam\'e systems in conical sets.

\begin{proposition}
\label{prop:lame-systems3}
Let $\mathcal{C}$ be the set given by 
$$\mathcal{C}:={\rm Int}\left(\bigcup_{i=1}^3\overline{\mathcal{C}_i}\right)$$
where $\mathcal{C}_i$, $i=1,2,3$,  are the circular sectors defined by
$$\mathcal{C}_i:=\{(x,y):\,x=\rho\cos(\theta),\,y=\rho\sin(\theta),\,\text{with }0< \rho< R,\,\text{ and }\theta_{i-1}<\theta<\theta_i\}$$
with $R>0$, and $0=:\theta_0\leq \theta_1<\theta_2\leq \theta_3<2\pi$ (see Figure \ref{trasmission}). 
Denote by 
$$\Gamma_{1,0}:=(0,R),$$
and
$$\Gamma_{3,0}:=\{(\rho\cos(\theta_3),\rho\sin(\theta_3))\in\R^2\,\text{with }0< \rho< R\},$$
the two external sides of $\mathcal{C}$, and by
$$\Gamma_i:=\{(x,y):\,x=R\cos\theta,\,y=R\sin\theta,\,\text{with }\theta_{i-1}\leq \theta<\theta_i\},$$
for $i=1,2,3$ the curvilinear portions of its boundary. Finally, consider the transmission interfaces 
$$\Gamma_{i,i+1}:=\partial \mathcal{C}_i\cap\partial \mathcal{C}_{i+1}\quad\text{ for }i=1,2.$$ 
We assume that each set $\mathcal{C}_i$ is the reference configuration of a linearly elastic, isotropic material whose Lam\'e coefficients are denoted by $\mu_i$ and $\lambda_i$, with $\mu_3:=\mu_1$ and $\lambda_3:=\lambda_1$, and satisfy the quasi-monotonicity condition:
$$\mu_2\geq \mu_1>0\quad\text{and}\quad\mu_2+\lambda_2\geq\mu_1+\lambda_1>0.$$

Let $(u_1,u_2, u_3)\in \displaystyle\prod_{i=1}^3 H^1(\mathcal{C}_i;\R^2)$  be a weak solution of the transmission problem:
\be{eq:transmission-problem}
\begin{cases}
\mu_i \Delta u_i+(\lambda_i+\mu_i)\nabla ({\rm div}\,u_i)= f_i&\text{in }\mathcal{C}_i,\,i=1,2,3,\\
 [2\mu_i Eu_i+\lambda_i ({\rm div }\,u_i)Id\,]\,\nu_{i,0}=0&\text{on }\Gamma_{i,0},\,i=1,3,\\
[2\mu_i Eu_i+\lambda_i ({\rm div }\,u_i)Id\,]\,\nu_i=g_i&\text{on }\Gamma_i,\,i=1,2,3,\\
u_i-u_{i+1}=0&\text{on }\Gamma_{i,i+1},\,i=1,2,\EEE\\
\Big[2\mu_i Eu_i-2\mu_{i+1} Eu_{i+1}&\\
\hspace{1.4cm}+\,\lambda_i ({\rm div }\,u_i)Id-\lambda_{i+1} ({\rm div }\,u_{i+1})Id\,\Big]\,\nu_{i,i+1}=0&\text{on }\Gamma_{i,i+1},\,i=1,2,
\end{cases}
\ee
where the data $f_i$ and $g_i$ satisfy $f_i\in L^2(\mathcal{C}_i)$, $g_i\in H^{1/2}(\Gamma_i,\R^2)$, $i=1,2,3$, the vectors $\nu_{i,i+1}$ are the normal to $\Gamma_{i,i+1}$ external to $ \mathcal{C}_{i}$, $i=1,2$, and the vectors $\nu_{1,0}$, $\nu_{3,0}$, and $\nu_i$ are the outer unit normals to $\Gamma_{1,0}$, $\Gamma_{3,0}$, and $\Gamma_i,\,i=1,2,3,$ respectively.   

If there exists a vector $\tau\in\Rz^2\setminus\{0\}$ such that $\tau\in\mathcal{C}_2$ and $-\tau\not\in\overline{\mathcal{C}}$, then there exists a neighbourhood $U$ of the origin such that
$$u\in H^{3/2+\ep}(U\cap \mathcal{C}_i)$$
for some $\ep>0$ and for $i=1,2,3$.
\end{proposition}

\begin{proof}
Let $\varphi\in C^{\infty}_c(\mathcal{C})$ be a cut-off function such that $0\leq \varphi\leq 1$ in $\mathcal{C}$, with ${\rm supp}\,\varphi\subset\subset B_{R/2}$, and $\varphi\equiv 1$ in $B_{R/3}$, where here $B_{R/2}$ and $B_{R/3}$ are the balls centered in the origin and with radii $R/2$ and $R/3$, respectively. Consider the maps $(\tilde{u}_1,\tilde{u}_2, \tilde{u}_3)\in \displaystyle\prod_{i=1}^3 H^1(\mathcal{C}_i;\R^2)$, defined as $\tilde{u}_i:=\varphi u_i$, $i=1,2,3.$
By straightforward computation, and in view of \eqref{eq:transmission-problem}, the triple $(\tilde{u}_1,\tilde{u}_2, \tilde{u}_3)$ solves the transmission problem
\be{eq:transmission-problemtilde}
\begin{cases}
\mu_i \Delta \tilde{u}_i+(\lambda_i+\mu_i)\nabla ({\rm div}\,\tilde{u}_i)= \tilde{f}_i&\text{in }\mathcal{C}_i,\,i=1,2,3,\\
[2\mu_i E\tilde{u}_i+\lambda_i ({\rm div }\,\tilde{u}_i)Id\,]\,\nu_{i,0}=\tilde{g}_i&\text{on }\Gamma_{i,0},\,i=1,3,\\
[2\mu_i E\tilde{u}_i+\lambda_i ({\rm div }\,\tilde{u}_i)Id\,]\,\nu_i=0&\text{on }\Gamma_i,\,i=1,2,3,\\
\tilde{u}_i-\tilde{u}_{i+1}=0&\text{on }\Gamma_{i,i+1},\,i=1,2,\\
\Big[2\mu_i E\tilde{u}_i-2\mu_{i+1} E\tilde{u}_{i+1}&\\
\hspace{1.2cm}+\,\lambda_i ({\rm div }\,\tilde{u}_i)Id-\lambda_{i+1} ({\rm div }\,\tilde{u}_{i+1})Id\,\Big]\,\nu_{i,i+1}=\tilde{h}_i&\text{on }\Gamma_{i,i+1},\,i=1,2,
\end{cases}
\ee
where $\tilde{f}_i\in L^2(\mathcal{C}_i)$, $\tilde{g}_i\in H^{1/2}(\Gamma_{i,0};\R^2)$, and $\tilde{h}_i\in H^{1/2}(\Gamma_{i,i+1};\R^2)$ for every $i$,  and the maps $\tilde{g}_i$ and $\tilde{h}_i$ vanish in the intersection of their domains with $B_{R/3}$. By applying Lemma \ref{lemma:grisvard} to both sets $\mathcal{C}_i$, $i=1,3$, with
\be{eq:deg-g}g=\begin{cases} 
\tilde{g}_i&\text{on }\Gamma_{i,0},\\
0&\text{on }\Gamma_i,\\
\tilde{h}_i&\text{on }\Gamma_{i,i+1},
\end{cases}\ee
we obtain functions $v_i\in H^2(\mathcal{C}_i;\R^2)$ such that
\be{eq:trace-ci}[2\mu Ev_i+\lambda ({\rm div }\,v_i)Id\,]\,\nu_{\mathcal{C}_i}=g\quad\text{on }\partial \mathcal{C}_i,\ee
where $\nu_{\mathcal{C}_i}$ is the outer unit normal to $ \mathcal{C}_i$ (where it exists), and
\be{eq:trace-vi}v_i=0\quad\text{on }\partial\mathcal{C}_i.\ee
Setting $(w_1,w_2,w_3):=(\tilde{u}_1-v_1,\tilde{u}_2, \tilde{u}_3-v_3)\in \displaystyle\prod_{i=1}^3 H^1(\mathcal{C}_i;\R^2)$, by \eqref{eq:transmission-problemtilde}, and \eqref{eq:deg-g}--\eqref{eq:trace-vi} there holds
\be{eq:transmission-problemfinal}
\begin{cases}
\mu_i \Delta w_i+(\lambda_i+\mu_i)\nabla ({\rm div}\,w_i)= \hat{f}_i&\text{in }\mathcal{C}_i,\,i=1,2,3,\\
 [2\mu_i Ew_i+\lambda_i ({\rm div }\,w_i)Id\,]\,\nu_{i,0}=0&\text{on }\Gamma_{i,0},\,i=1,3,\\
[2\mu_i Ew_i+\lambda_i ({\rm div }\,w_i)Id\,]\,\nu_i=0&\text{on }\Gamma_i,\,i=1,2,3,\\
w_i-w_{i+1}=0&\text{on }\Gamma_{i,i+1},\,i=1,2,\EEE\\
\Big[2\mu_i Ew_i-2\mu_{i+1} Ew_{i+1}&\\
\hspace{1.4cm}+\,\lambda_i ({\rm div }\,w_i)Id-\lambda_{i+1} ({\rm div }\,w_{i+1})Id\,\Big]\,\nu_{i,i+1}=0&\text{on }\Gamma_{i,i+1},\,i=1,2,
\end{cases}
\ee
where $\hat{f}_i\in L^2(\mathcal{C}_i)$ for $i=1,2,3$. By \cite[Theorem 2]{KS} we obtain that there exists a neighborhood $\tilde{U}$ of the origin such that $w_i\in H^{3/2+\ep}(\tilde{U}\cap \mathcal{C}_i)$ for $i=1,2,3$. The thesis follows by observing that on $U:=\tilde{U}\cap B_{R/3}$, the triple $(u_1,u_2, u_3)$ satisfies
$$(u_1,u_2, u_3)=(\tilde{u}_1,\tilde{u}_2, \tilde{u}_3)=(w_1+v_1,w_2,w_3+v_3),$$
and by the regularity of the maps $v_i$, $i=1,3$.
\end{proof}

\begin{figure}[htp]
\begin{center}
\includegraphics[scale=0.8]{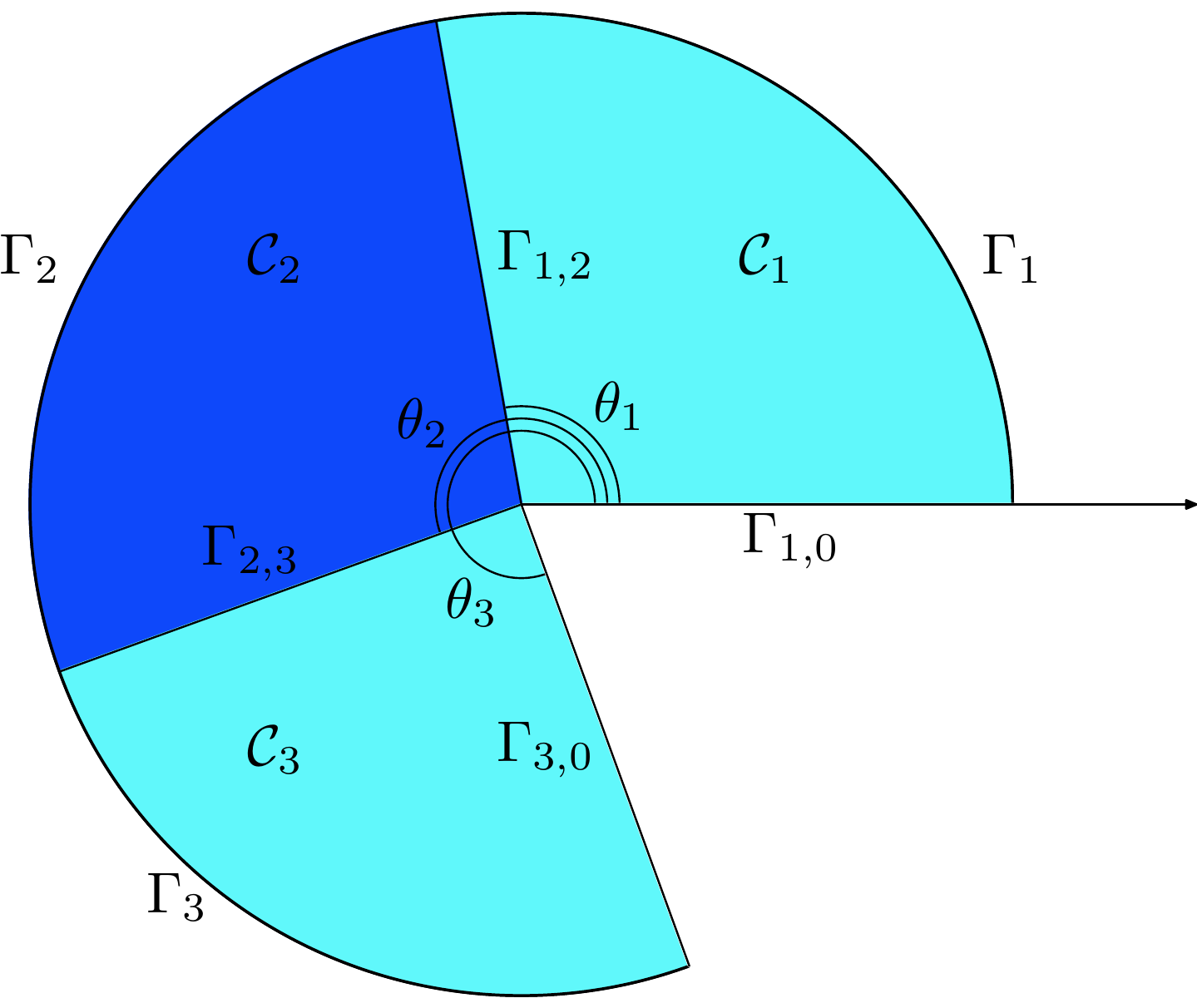}
\caption{The geometry of the set $\mathcal{C}$ on which we consider the transmission problem of Proposition \ref{prop:lame-systems3} is depicted. The lines $\Gamma_{1,2}$ and $\Gamma_{2,3}$ are transmission interfaces for such problem.}
\label{trasmission}
\end{center}
\end{figure}

We are now ready to provide a decay estimate for the gradient of minimizing displacements at the points in which the graph of the corresponding minimizing profile is locally Lipschitz.

\begin{proposition}[Decay estimate]
\label{thm:blow}
Let $(u, h)\in X$ be a $\mu$-local minimizer for the functional $\mathcal{F}$ and assume that the Lam\'e coefficients of film and substrate satisfy the monotonicity condition \eqref{monotonicity}. Let $z_0:=(x_0,h(x_0))\in \Gamma_h\setminus (\Gamma_h^{cut}\cup \Gamma_h^{cusp})$. Then there exists a constant $C>0$, a radius $r_0$, and an exponent $\frac12<\alpha<1$, such that
$$\int_{B(z_0,r)\cup\Omega_h}|\nabla u|^2\,dx\,dy\leq Cr^{2\alpha}$$
for all $0<r<r_0$.
\end{proposition}

\begin{proof}

We begin by considering the case in which $h(x_0)=0$. If there exists a constant $C$ such that
$$\limsup_{r\to 0}\frac{1}{r^2}\int_{B(z_0,r)\cap\Omega_h}|\nabla u|^2\,dx\,dy\leq C,$$
then there is nothing to prove. Thus, we assume that this does not hold and that there exists a sequence $\{r_n\}\subset \R$ such that $r_n\to 0$ and
\be{eq:bu-contradiction}
\limsup_{n\to +\infty}\frac{1}{r_n^2}\int_{B(z_0,r_n)\cap\Omega_h}|\nabla u|^2\,dx\,dy=+\infty.
\ee
We subdivide the proof into three steps.\\

\noindent\textbf{Step 1}: We claim that there exist an orthonormal basis $\{v_1,v_2\}$ of $\R^2$, three constants $C_1,L>0$, $\tau_0\in (0,1)$, and an exponent $\frac12<\beta<1$ such that for all $\tau\in (0,\tau_0)$ there exists a radius $0<r_{\tau}<1$ satisfying
\be{eq:bu-Step1-to-prove}
\int_{C(z_0,\tau r_n)}|\nabla u|^2\,dx\,dy\leq C_1 \tau^{2\beta}\int_{C(z_0,r_n)}(1+|\nabla u|^2)\,dx\,dy,
\ee
for all $0<r_n<r_{\tau}$, where
$$C(z_0,r_n):=\Omega_h \cap\{z_0+sv_1+tv_2:\,-r_n<s<r_n,\,-4Lr_n<t<4Lr_n\}.$$
We point out that, once claim \eqref{eq:bu-Step1-to-prove} is proved, the assert of the theorem follows arguing as in \cite[Theorem 3.13, Step 6]{FFLM2}.

To prove \eqref{eq:bu-Step1-to-prove} we first observe that, since $z_0\in \Gamma_h^{reg}$, we can apply Proposition \ref{prop:loc-lip} to obtain a Lipschitz function $g:(-a',a')\to (-b',b')$ with ${\rm Lip }\,g\leq L$ for some $L>1$ such that $g(0)=0$, and
$$\Omega_h\cap Q=\{z_0+sv_1+tv_2:\,-a'<s<a',\,-b'<t<g(s)\},$$
where
$$Q=\{z_0+sv_1+tv_2:\,-a'<s<a',\,-b'<t<b'\}.$$
Note that $g$ has left (right) derivative in every point that is left (right) continuous. By Korn inequality in Lipschitz domains we deduce that $u\in H^1(\Omega_h\cap Q;\R^2)$. If $r_n\leq \min\{a',\tfrac{b'}{4L}\}$, then $C(z_0,r_n)\subset Q\cap \Omega_h$. Therefore,
$$C(z_0,r_n)=\{z_0+sv_1+tv_2:\,-r_n<s<r_n,\,-4L r_n<t<g(s)\}.$$
Fix $C_1>0$, $\tau_0\in (0,1)$ and $\beta>\frac12$ to be determined later, and assume by contradiction that \eqref{eq:bu-Step1-to-prove} is false for some $\tau\in (0,\tau_0)$. Up to the extraction of a (non-relabeled) subsequence there holds
\be{eq:Step1-cons1}
\int_{C(z_0,\tau r_n)}|\nabla u|^2\,dx\,dy>C_1\tau^{2\beta}\int_{C(z_0,r_n)}(1+|\nabla u|^2)\,dx\,dy,
\ee
for a sequence $r_n\to 0$.
Define the sets
$$C_n:=\frac{1}{r_n}(-z_0+C(z_0,r_n))=\left\{sv_1+tv_2:\,-1<s<1,\,-4L<t<\frac{g(r_n s)}{r_n}\right\}.$$
We have
\be{eq:bu-conv-sets}\chi_{C_n}\to \chi_{C_{\infty}}\quad\text{a.e. in }\R^2,\ee
where
$$C_{\infty}:=\{sv_1+tv_2:\,-1<s<1,\,-4L<t<g_{\infty}(s)\},$$
the function $g_{\infty}$ is defined as
$$g_{\infty}(s):=\begin{cases}g'_{-}(0)s&\text{for }s<0,\\
g'_{+}(0)s&\text{for }s>0,\end{cases}$$
and $\chi_{C_n}$ and $\chi_{C_{\infty}}$ are the characteristic functions of the sets $C_n$ and $C_{\infty}$, respectively.

Define the maps
$$u_n(z):=\frac{u(z_0+r_n z)-a_n}{\lambda_n r_n},\quad \text{for every }z\in C_n,$$
where
\be{eq:bu-def-quant}
a_n:=\frac{1}{|C(z_0,r_n)|}\int_{C(z_0,r_n)}u(x,y)\,dx\,dy,\quad\lambda_n^2:=\frac{1}{|C(z_0,r_n)|}\int_{C(z_0,r_n)}|\nabla u|^2\,dx\,dy.\ee
We point out that
\be{eq:bu-w12-un}\frac{1}{|C_n|}\int_{C_n}|\nabla u_n|^2\,dz=\frac{1}{\lambda_n^2|C(z_0,r_n)|}\int_{C(z_0,r_n)}|\nabla u|^2\,dx\,dy=1\ee
and
\begin{align*}
&\int_{C_n}u_n\,dz=\frac{1}{\lambda_n r_n}\int_{C_n}u(z_0+r_n z)\,dx-\frac{a_n|C_n|}{\lambda_n r_n}\\
&\qquad=\frac{1}{\lambda_n r_n^3}\Big(\int_{C(z_0,r_n)}u\,dx\,dy-{a_n|C(z_0,r_n)|}\Big)=0.
\end{align*}
Extend the maps $u_n$ to the rectangle 
$$R:=\{sv_1+tv_2:\,-1<s<1,\,-4L<t<4L\}$$
so that $u_n\in W^{1,2}(R;\R^2)$. By \eqref{eq:bu-w12-un} we obtain the uniform bound
$$\|u_n\|_{W^{1,2}(R;\R^2)}\leq C\|\nabla u_n\|_{L^2(C_n;\M^{2\times 2})}\leq C.$$
Thus, there exist $u_{\infty}\in W^{1,2}(R;\R^2)$, and $\lambda_{\infty}\in [0,+\infty]$ such that, up to the extraction of a (non-relabelled) subsequence, there holds
\be{eq:bu-w-un}u_n\wk u_{\infty}\quad\text{weakly in }W^{1,2}(R;\R^2),\ee
and
\be{eq:bu-conv-lambdan}
\lambda_n\to \lambda_{\infty}.
\ee
%By \eqref{eq:Step1-cons1} and \eqref{eq:bu-def-quant}, and since $\tau_0<1$, we obtain
%$$\lambda_n^2>C_1\tau^{2\beta}(1+\lambda_n^2),$$
%thus by \eqref{eq:bu-conv-lambdan}, we deduce that $\lambda_{\infty}>0$. 
In addition,
\begin{align*}
&\frac{1}{r_n^2}\int_{B(x_0,r_n)\cap \Omega_h}|\nabla u|^2\,dx\,dy\leq \frac{1}{r_n^2}\int_{C(x_0,r_n)}|\nabla u|^2\,dx\,dy=\lambda_n^2 \frac{|C(z_0,r_n)|}{r_n^2}\\
&\quad=\lambda_n^2 |C_n|\leq 12 L \lambda_n^2.
\end{align*}
Hence, by \eqref{eq:bu-contradiction} we conclude that 
\be{eq:bu-lambda-infty}\lambda_{\infty}=+\infty.\ee

In view of a change of variable, the maps $u_n$ satisfy the Euler-Lagrange equations
$$\int_{C_n}E\varphi(z):\C(r_n z_2)Eu_n(z)\,dz=\frac{1}{\lambda_n}\int_{C_n}E\varphi(z):\C(r_n z_2)E_0(r_n z_2)\,dz$$
for every $\varphi\in C^1_0(R;\R^2)$. 
Thus, by \eqref{eq:bu-conv-sets}, \eqref{eq:bu-w-un}, \eqref{eq:bu-conv-lambdan}, and \eqref{eq:bu-lambda-infty} we deduce that
\be{eq:bu-limit-Euler}
\int_{C_{\infty}}E\varphi(z):\C(z_2)Eu_{\infty}(z)\,dz=0\quad\text{for every }\varphi\in C^1_0(R;\R^2).
\ee

\textbf{Step 2}:
Fix a ball $B$ such that
$$B\subset\subset \{sv_1+tv_2:\,-1<s<1,\,-4L<t<-3L\}.$$
We claim that 
\be{eq:bu-claim}
\lim_{n\to +\infty}\int_{C_n}\psi^2|\nabla u_n-\nabla u_{\infty}|^2\,dz=0
\ee
for every $\psi\in C^1_0(R)$ vanishing in $B$. Arguing as in \cite[Theorem 3.13, Step 2]{FFLM2} we obtain that
$$\lim_{n\to +\infty}\int_{C_n}\psi^2(z) (Eu_n(z):\C(r_nz_2)Eu_n(z)-Eu_{\infty}(z):\C(z_2)Eu_{\infty}(z))\,dz=0,$$
hence
$$\lim_{n\to +\infty}\int_{C_n}E(\psi(z)(u_n(z)-u_{\infty}(z))):\C_{\infty}(z)E(\psi(z)(u_n(z)-u_{\infty}(z)))\,dz=0.$$
Claim \eqref{eq:bu-claim} follows then from Korn's inequality (see \cite[Theorem 4.2]{FFLM2}).\\

\textbf{Step 3}: by Step 1, we deduce that $u_{\infty}$ is a weak solution of the transmission problem
\begin{subequations}
\label{eq:transmission-problem}
\begin{align}
\nn & \mu_f \Delta u_{\infty}^+ +(\lambda_f+\mu_f)\nabla ({\rm div}\,u_{\infty}^+)=0&\text{in }C_{\infty}^+,\\
\nn & \mu_s \Delta u_{\infty}^- +(\lambda_s+\mu_s)\nabla ({\rm div}\,u_{\infty}^-)=0&\text{in }C_{\infty}^-,\\
\nn & (2\mu_f Eu_{\infty}^+ +\lambda_f ({\rm div }\,u_{\infty}^+)Id)\nu_{\infty}=0&\text{on }\Gamma_{g_{\infty}},\\
\nn & \label{eq:u1} u_{\infty}^+-u_{\infty}^-=0&\text{on }\{z_2=0\},\\
\nn &(2\mu_f Eu_{\infty}^+-2\mu_s Eu_{\infty}^- +\lambda_f ({\rm div }\,u_{\infty}^+)Id-\lambda_f ({\rm div }\,u_{\infty}^-)Id)e_2=0&\text{on }\{z_2=0\},
\end{align}
\end{subequations}
where $C_{\infty}^+:=C_{\infty}\cap \{z_2>0\}$, $C_{\infty}^-:=C_{\infty}\cap \{z_2<0\}$, $\Gamma_{g_{\infty}}:=\{(s,g_{\infty}(s)):-1<s<1\}$, $u_{\infty}^+:=u_{\infty}|_{C_{\infty}^+}$, $u_{\infty}^-:=u_{\infty}|_{C_{\infty}^-}$, and $\nu_{\infty}$ is the outer unit normal to $\Gamma_{g_{\infty}}$, wherever it exists. Note that the fourth condition in \eqref{eq:transmission-problem} holds because Sobolev maps are absolutely continuous on almost every line, whereas the other equations in \eqref{eq:transmission-problem} are a consequence of \eqref{eq:bu-limit-Euler}.

In view of \eqref{monotonicity} and the geometry of the problem we can apply Proposition \ref{prop:lame-systems3} with $R\leq 1$ to $u_{\infty}$, with $\theta_2=\pi$, $\mu_1=\mu_f$, $\lambda_1=\lambda_f$, $\mu_2=\mu_s$, $\lambda_2=\lambda_s$, and with data $f_i=0$, and 
$$g_i:=[2\mu_i Eu_{\infty}+\lambda_i ({\rm div }\,u_{\infty})Id\,]\,\nu_i\quad\text{on }\Gamma_i,\,i=1,2,3,$$
where $\Gamma_1=\partial B(O,R)\cap C_{\infty}^+\cap\{z_1<0\}$, $\Gamma_2=\partial B(O,R)\cap C_{\infty}^-$, and $\Gamma_3=\partial B(O,R)\cap C_{\infty}^+\cap\{z_1>0\}$ where $O$ denotes the origin $(0,0)$.
Therefore, we conclude that there exists a ball $B\subset B(O,R)$ centered in the origin, and such that
$$u_{\infty}\in H^{3/2+\ep}\big(B\cap \big(C_{\infty}^+\cup C_{\infty}^-\big);\R^2\big).$$
Thus, by H\"older inequality we obtain
\begin{align*}
&\int_{B(O,r)\cap C_{\infty}}|\nabla u_{\infty}|^2\,dz=\int_{B(O,r)\cap C_{\infty}^+}|\nabla u_{\infty}|^2\,dz+\int_{B(O,r)\cap C_{\infty}^-}|\nabla u_{\infty}|^2\,dz\\
&\quad\leq r^{2-\frac4s}\big(\|\nabla u\|^2_{L^s({B}\cap C_{\infty}^+;\M^{2\times 2})}+\|\nabla u\|^2_{L^s({B}\cap C_{\infty}^-;\M^{2\times 2})}\big)\leq Cr^{2\beta}
\end{align*}
for every $r>0$ small enough, where $4<s<\frac{4}{1-2\ep}$, and $\beta=1-\frac{2}{s}>\frac12$, where we used the fact that 
$$H^{3/2+\ep}\big(B\cap \big(C_{\infty}^+\cup C_{\infty}^-\big);\R^2\big)\subset L^s\big(B\cap \big(C_{\infty}^+\cup C_{\infty}^-\big);\R^2\big)$$ 
for every $s\in \Big[1,\frac{4}{1-2\ep}\Big]$.
Choosing $\tau_0$ such that
$$\tau_0 C_{\infty}\subset (B(O,1)\cap C_{\infty})\setminus \{se_1+te_2:\,-1<s<1,\,-4L<t<-3L\},$$
by Step 2 we deduce that for $0< \tau\leq \tau_0$ there holds
\begin{align*}
&\lim_{n\to +\infty}\frac{\int_{C(z_0,\tau r_n)}|\nabla u|^2\,dx\,dy}{\int_{C(z_0,r_n)}|\nabla u|^2\,dx\,dy}=\frac{1}{|C_{\infty}|}\lim_{n\to +\infty}\int_{\tau C_n}|\nabla u_n|^2\,dz\\
&=\frac{1}{|C_{\infty}|}\int_{\tau C_{\infty}}|\nabla u_{\infty}|^2\,dz\leq \frac{1}{|C_{\infty}|}\int_{B\big(O,\frac{\tau}{\tau_0}\big)\cap C_{\infty}}|\nabla u_{\infty}|^2\,dz\leq \frac{C_2 \tau^{2\beta}}{|C_{\infty}|}.
\end{align*}
This leads to a contradiction to \eqref{eq:Step1-cons1} provided that $C_1\geq \frac{C_2 \tau^{2\beta}}{|C_{\infty}|}$, and thus completes the proof of \eqref{eq:bu-Step1-to-prove} in the case $h(x_0)=0$.
The same argument works for $h(x_0)>0$ by noticing that in this latter scenario after the blow-up $\C_{\infty}(h(x_0)+r_n\cdot)\equiv \C_f$ (see also \cite[Theorem 3.13]{FFLM2}). 
\end{proof}

\begin{comment}
The previous Decay estimate can be improve on the portions of $\Gamma_h$ where $h$ is $C^1$.

\begin{proposition}
\label{pro:blow2}
Let $(u, h)\in X$ be a $\mu$-local minimizer for the functional $\mathcal{F}$ and assume that the Lam\'e coefficients satisfy the monotonicity condition \eqref{monotonicity}.  Then for every portion $\Gamma'\subset\Gamma_h\setminus(\Gamma_h^{cut}\cup\Gamma_h^{cusp})$ that is $C^1$-regular and for every parameter $0<\alpha<1$ there exist a constant $c>0$ and a radius $r_0$ such that
$$\int_{B(z_0,r)\cup\Omega_h}|\nabla u|^2\,dx\,dy\leq Cr^{2\alpha}$$
for  all $z_0\in\Gamma'$ and $0<r<r_0$. %\Gamma_h\setminus(\Gamma_h^{cut}\cup\Gamma_h^{cusp})
\end{proposition}

\end{comment}

\section{Contact-Angle conditions}  \label{sec:YD}  
This section is devoted to the proof of Theorem \ref{thm:YDlaw}. For every profile function $h$ we denote by $h'_{-}(x)$ and $h'_{+}(x)$, respectively, the left and right derivative of $h$ in $x\in[a,b]$, whenever they exist. In the following we denote by $\theta^*$ the angle
\be{eq:def-theta-star-prel}\theta^*:=\arccos{\beta},\ee
where $\beta$ is the quantity defined in \eqref{eq:beta}. We first provide a preliminary characterization of contact-angle conditions.
\begin{proposition}
\label{pro:YDlawold} Assume that the  Lam\'e coefficients of the film and the substrate satisfy %\emph{quasi-monotonicity} condition:
\be{monotonicity}\mu_s\geq \mu_f>0\quad\text{and}\quad\mu_s+\lambda_s\geq\mu_f+\lambda_f>0.
\ee
%Let the film and the substrate be isotropic materials satisfying the \emph{quasi-monotonicity} condition:
%\be{monotonicity}\mu_s\geq \mu_f\quad\text{and}\quad\mu_s+\lambda_s\geq\mu_f+\lambda_f.
%\ee
%where  $\mu_f$, $\lambda_f$, and $\mu_s$, $\lambda_s$ are their corresponding Lam\'e coefficients.

Then, for every $\mu$-local minimizer $(u,h)\in X$ of $\mathcal{F}$ and for $z_0:=(x_0,0)\in Z_h\cap\Gamma_h^{reg}$ the following asserts hold true:
\begin{enumerate}
%\item[1.] If $h(x_0)>0$, then $\omega(z_0)=\pi$.
\item[1.] For every $x_0\in P_h$ we have that  $\theta^{-}(x_0),\theta^+(x_0)\in [0,\theta^*]$ and, if $\theta^{-}(x_0)=\theta^+(x_0)$ then $\theta^{-}(x_0)=\theta^+(x_0)=0$,
\item[2.] For any $(c,d)\in I_h$, there holds $\theta^{-}(c),\theta^+(d)\in [0,\theta^*]$.
\end{enumerate}
 Additionally, $\Gamma_h^{jump}$ satisfies the following property
\begin{enumerate}
%\item[1.] If $h(x_0)>0$, then $\omega(z_0)=\pi$.
\item[3.]  If $\beta\neq 0$, then  $\Gamma_h^{jump}\cap Z_h= \emptyset$.
%\item[4.]  If $\beta=0$, then every connected component of $\Gamma_h^{jump}$ intersects the $x$-axis.
%\item[5.] if $\beta<1$, then $P_h=\emptyset$.
\end{enumerate}
\end{proposition}

\begin{proof}
Let $(u,h)$ be a $\mu$-local minimizer of $\mathcal{F}$, and let $z_0=(x_0,h(x_0))\in Z_h\cap (\Gamma_h^{reg}\cup \Gamma_h^{jump})$.  As a consequence of Assertion 1. of Proposition \ref{prop:loc-lip} there exist $a'>0$ and $b'>0$ such that the function $g:(-a',a')\to (-b',b')$ defined as 
$$g(x):=h(x)-h(x_0)\quad\text{for every }x\in (-a',a')$$ 
satisfies one of the following conditions:
\begin{itemize}
\item[($c_1$)] $g$ is a Lipschitz function in $(-a',a')$ with Lipschitz constant ${\rm Lip}\, g\leq L$ for some $L>1$;
\item[($c_2$)] $g$ is a  Lipschitz function in $(-a',0)$ with Lipschitz constant ${\rm Lip}\, g\leq L$ for some $L>1$, and $g'_{+}(0)=\infty$;
\item[($c_3$)] $g$ is a Lipschitz function in $(0,a')$ with Lipschitz constant ${\rm Lip}\, g\leq L$ for some $L>1$, and $g'_{-}(0)=-\infty$.
\end{itemize}
\noindent  We also point out that in view of the internal-ball condition (see Proposition \ref{thm:palla-interna}), under condition ($c_1$), the angle between $g_{-}'(0)$ and $g'_{+}(0)$ intersecting $\Omega_h^-$ is always in the interval $[\pi,2\pi)$.

In the following we denote the intersection of a given a set $U$ with the half-planes $\{x<0\}$ and $\{x>0\}$ by  $U^{2}:=U\cap\{x<0\}$ and $U^{3}:=U\cap\{x>0\}$, respectively. We also set $U^1:=U$.

Choose an infinitesimal sequence $r_n\to 0$, and consider the sets
$$C(z_0,r_n):=\{z_0+(x,y)\in \R^2:\,-r_n<x<r_n,\,-4L r_n<y<g(x)\},$$ 
and 
$$C_n:=\frac{1}{r_n}(C(z_0,r_n)-z_0)=\left\{z\in \R^2:-1<z_1<1, -4L<z_2<\frac{g(r_n z_1)}{r_n}\right\}.$$
We observe that for $k=1,2,3$ we have that $\chi_{C_n^k}\to \chi_{C_{\infty}^k}$ a.e., where 
$$C_{\infty}:=\left\{z\in \R^2:-1<z_1<1, -4L<z_2<g_{\infty}(z_1)\right\},$$
the function $g_{\infty}:(-1,1)\to \R$ is defined as
$$g_{\infty}(z_1):=\begin{cases}
g'_{-}(0)z_1&\text{ for }z_1<0,\\
g'_{+}(0)z_1&\text{ for }z_1>0.
\end{cases},$$
and $\chi_{C_n^k},\, \chi_{C_{\infty}^k}$ denote the characteristic functions of the sets $C_n^k$ and $C_{\infty}^k$, respectively.
In particular, $C_{\infty}^k\subset R^k$ for $k=1,2,3$, where
$$R:=\{z\in \R^2:-1<z_1<1,\,-4L<z_2<4L\}.$$
With a slight abuse of notation under each condition ($c_k$) we identify the map $u$ with its $H^1$-extension to the set $\Omega_h \cup R^k(z_0,r_n)$, where
$$R(z_0,r_n):=z_0+r_n R.$$
Note that this extension is well-defined owing to Assertion 1. of Proposition \ref{prop:loc-lip}, which guarantees that the graph of $h$, aside from cusps and cuts, is locally Lipschitz.

 Let $I^k$ be defined as
 $$I^k:=\begin{cases}
 (-1,1)&\text{if }k=1,\\
 (-1,0)&\text{if }k=2,\\
 (0,1)&\text{if }k=3.
 \end{cases}$$
For every $0<\delta<1$ under each condition ($c_k$) we consider a function $\psi_{\delta}\in W^{1,\infty}(I^k)$ to be specified later, satisfying the following properties
\begin{align}
&\label{P1add}\sup_{\delta}\|\psi_{\delta}\|_{L^{\infty}(I^k)}\leq C,\\
&\label{P2add}0\leq g_{\infty}(z_1)+\delta \psi_{\delta}(z_1)< 4L\quad\text{for every }z_1\in I^k,\\
&\label{P3add}{\rm supp }\,\psi_{\delta}=[x_{\delta}^-,x_{\delta}^+],
\end{align}
with
\be{supppsidelta}\begin{cases}
x_{\delta}^-<0<x_{\delta}^+&\text{for condition }(c_1),\\
x_{\delta}^-<0\text{ and }x_{\delta}^+=0&\text{for condition }(c_2),\\
x_{\delta}^-=0\text{ and }x_{\delta}^+>0&\text{for condition }(c_3).
\end{cases}\ee

Define the maps
$$\psi_{\delta}^n(x):=\begin{cases} r_n\psi_{\delta}\left(\frac{x-x_0}{r_n}\right)&\text{for every }x\in [x_0+r_n x_{\delta}^-,x_0+r_nx_{\delta}^+],\\
0&\text{otherwise in }(a,b).\end{cases}$$
Note that for $r_n$ small enough,
$$\Omega_{h+\delta\psi_{\delta}^n}\subset \Omega_h\cup R^k(z_0,r_n),$$
and
\be{eq:star35}|\Omega_{h+\delta\psi_{\delta}^n}\Delta \Omega_h|\leq \frac{\mu}{2}.\ee

 By Proposition \ref{prop:penalization} there exists $\lambda_0>0$ such that
$$\mathcal{F}(u, h)=\min\left\{\mathcal{F}(v, \tilde{h})+\lambda| |\Omega_h^+|-|\Omega_{\tilde{h}}^+||:\,(v,\tilde{h})\in X,\,|\Omega_{\tilde{h}}\Delta \Omega_h|\leq \frac{\mu}{2}\right\}$$
for all $\lambda\geq \lambda_0$. In the following we denote by $\mathcal{G}$ the volume-penalized functional defined as
\be{eq:energyg}\mathcal{G}(v, \tilde{h}):=\mathcal{F}(v, \tilde{h})+\lambda_0||\Omega_h^+|-|\Omega_{\tilde{h}}^+||\ee
for every $(v,\tilde{h})\in X$. By the minimality of $(u,h)$, and by \eqref{filmenergyfinal}, \eqref{eq:cuts}, \eqref{eq:beta}, \eqref{eq:star35}, and \eqref{eq:energyg}, there holds
 \be{eq:el}0\leq \frac{\mathcal{G}(u,h+\delta\psi_{\delta}^n)-\mathcal{G}(u,h)}{\delta r_n}:=A^1_n-A^2_n+B_n+D_n+E_n,\ee
 where
 \begin{align}
 &\label{eq:A1}A^1_n:=\frac{1}{\delta r_n}\int_{\Omega_{h+\delta\psi_{n}}\setminus\Omega_h}W_0(y,Eu(x,y)-E_0(y))\,dx\,dy,\\
 &\label{eq:A2}A^2_n:=\frac{1}{\delta r_n}\int_{\Omega_h\setminus\Omega_{h+\delta\psi_{\delta}^n}}W_0(y,Eu(x,y)-E_0(y))\,dx\,dy,\\
 &\label{eq:B}B_n:=\frac{\gamma_f\beta}{\delta r_n}(\HH^1(\tilde{\Gamma}_{h+\delta\psi_{\delta}^n}\cap\{y=0\})-\HH^1(\tilde{\Gamma}_{h}\cap\{y=0\}))\\
 &\nonumber\quad+\frac{\gamma_f}{\delta r_n}(\HH^1(\tilde{\Gamma}_{h+\delta\psi_{\delta}^n}\cap\{y>0\})-\HH^1(\tilde{\Gamma}_{h}\cap\{y>0\})),\\
 &\label{eq:D}D_n:=\frac{\lambda_0}{\delta r_n}\left|\int_a^b h-(h+\delta\psi_{\delta}^n)\,dx\right|,
 \end{align}
 and
 \be{eq:E}E_n:=\frac{2\gamma_f}{\delta r_n}\sum_{x\in C(h)}((h+\delta\psi_{\delta}^n)^-(x)-(h+\delta\psi_{\delta}^n)(x))-\frac{2\gamma_f}{\delta r_n}\sum_{x\in C(h)}(h^-(x)-h(x)).\ee
 We begin by noticing that
\be{eq:Ezero}E_n=0\ee
 by the regularity of $\psi_{\delta}^n$, and that
 \be{eq:dn}
 D_n\leq\frac{\lambda_0}{r_n}\int_{r_n x_{\delta}^- +x_0}^{r_n x_{\delta}^+ +x_0} |\psi_{\delta}^n(x)|\,dx=\lambda_0 r_n\int_{x_{\delta}^-}^{x_{\delta}^+}|\psi_{\delta}(z_1)|\,dz_1\to 0
 \ee
 as $n\to +\infty$, by the change of variable 
 \be{variablechange}
 x=x_0+r_n z_1.
 \ee

 \noindent\textbf{Step 1} (Convergence of the elastic-energy terms){.}  We show that $A^1_n\to 0$. The proof that $A^2_n\to 0$ is analogous. We begin by assuming that the quantities
 $$\lambda_n:=\frac{1}{r_n}\left(\int_{C(z_0,r_n)}|\nabla u|^2\,dx\,dy\right)^{\tfrac12}$$
satisfy
\be{lamdaconditions}\limsup_{n\to +\infty} \lambda_n<+\infty.
\ee
In this situation we define the maps $v_n:C_n^k\to\R^2$, as
$$v_n(z):=\frac{u(z_0+r_nz)-\fint_{C^k(z_0,r_n)}u(x,y)\,dx\,dy}{r_n}.$$
Notice that by construction we have $\int_{C_n^k}v_n(x,y)\,dx\,dy=0$. Since $u\in H^1(\Omega_h\cup R(z_0,r_n)^k;\R^2)$, in each case ($c_k$), $k=1,2,3$, the map $v_n$ satisfies $v_n\in H^1(R^k;\R^2)$, $k=1,2,3$, and 
\be{eq:boundvn} \|v_n\|_{W^{1,2}(R^k;\R^2)}\leq C\|\nabla v_n\|_{L^2(C_n^k;\M^{2\times 2})}\leq C\lambda_n^2\leq C\ee
for $n$ big enough, where the last inequality follows from \eqref{lamdaconditions}.
%Thus, there exists $v_{\infty}\in W^{1,2}(R;\R^2)$ such that, up to the extraction of a (not relabeled) subsequence we have $$v_n\wk v_{\infty}\quad\text{weakly in }W^{1,2}(R;\R^2).$$
Therefore for each case ($c_k$), $k=1,2,3$ we conclude that
\begin{align*}
A^1_n&=\frac{1}{\delta r_n}\int_{I_n^{\delta}}\int_{h(x)}^{h(x)+\delta\psi_{\delta}^n(x)}\C_f(Eu(x,y)-E_0(y)):(Eu(x,y)-E_0(y))\,dy\,dx\\
&=\frac{r_n}{\delta}\int_{I^{\delta}}\int_{\frac{h(x_0+r_n z_1)}{r_n}}^{\frac{h(x_0+r_n z_1)}{r_n}+\delta\psi_{\delta}(z_1)}\C_f(Ev_n(z)-E_0(z_2)):(Ev_n(z)-E_0(z_2))\,dz_2\,dz_1\\
&\leq \frac{C r_n}{\delta}\|E v_n-E_0\|_{L^2(R^k;\M^{2\times 2}_{\rm sym})}\leq \frac{C r_n}{\delta},
\end{align*}
with $I_n^{\delta}:=(x_0+r_nx_{\delta}^-, x_0+r_nx_{\delta}^+)\cap \{\psi_{\delta}^n\geq 0\}$, $I^{\delta}:=(x_{\delta}^-, x_{\delta}^+)\cap \{\psi_{\delta}\geq 0\}$, where in the second equality we performed the change of variable 
\be{eq:changevariable}(x,y)=(x_0+r_nz_1,r_nz_2),
\ee
and where the last inequality follows from \eqref{eq:boundvn}, and \eqref{P1add}--\eqref{supppsidelta}. 

When \eqref{lamdaconditions} does not hold, we have
$$\limsup_{n\to +\infty} \lambda_n=+\infty.$$
Then, in view of Proposition \ref{thm:blow} there holds
\be{eq:case2limit} r_n\lambda_n^2=\frac{1}{r_n}\int_{C(z_0,r_n)}|\nabla u|^2\,dx\,dy\leq Cr_n^{2\alpha-1}\to 0.\ee

We define the maps $w_n:C_n^k\to\R^2$, as
$$w_n(z):=\frac{u(z_0+r_nz)-\fint_{C^k(z_0,r_n)}u(x,y)\,dx\,dy}{\lambda_n r_n},$$
for every $z\in C_n^k$. Note that $\int_{C_n^k}w_n(x)\,dx=0$ by construction. 
Again, the fact that $u\in H^1(\Omega_h\cup R^k(z_0,r_n);\Rz^2)$ implies that $w_n\in H^1(R^k;\Rz^2)$, with 
\be{eq:boundun} \|w_n\|_{W^{1,2}(R^k;\R^2)}\leq C\|\nabla w_n\|_{L^2(C_n^k;\M^{2\times 2})}\leq C.\ee
%Thus, there exists $u_{\infty}\in W^{1,2}(R;\R^2)$ such that, up to the extraction of a (not relabeled) subsequence we have $$w_n\wk u_{\infty}\quad\text{weakly in }W^{1,2}(R;\R^2).$$
By employing the same change of variable \eqref{eq:changevariable} of the first case we  observe that 
\begin{align*}
A^1_n&=\frac{1}{\delta r_n}\int_{I^{\delta}_n}\int_{h(x)}^{h(x)+\delta\psi_{\delta}^n(x)}\C_f(Eu(x,y)-E_0(y)):(Eu(x,y)-E_0(y))\,dy\,dx\\
&=\frac{\lambda_n^2 r_n}{\delta}\int_{I^{\delta}}\int_{\frac{h(x_0+r_n z_1)}{r_n}}^{\frac{h(x_0+r_n z_1)}{r_n}+\delta\psi_{\delta}(z_1)}\C_f(Ew_n(z){-}E_0(z_2)){:}(Ew_n(z){-}E_0(z_2))\,dz_2\,dz_1\\
&\leq \frac{C \lambda_n^2 r_n}{\delta}\|E w_n-E_0\|_{L^2(R^k;\M^{2\times 2}_{\rm sym})}\leq \frac{C}{\delta} \lambda_n^2 r_n
\end{align*}
where now in the last inequality we used \eqref{eq:boundun} and \eqref{P1add}--\eqref{supppsidelta}. The claim follows from \eqref{eq:case2limit}.

 \noindent\textbf{Step 2} (Surface-energy convergence under condition ($c_1$))\textbf{.} In this step we study the convergence of the terms $B_n$  under condition ($c_1$). To this aim, we treat in three different subsections the cases of island borders, of valleys with no vanishing contact angles, and of valleys with one vanishing contact angle. In particular, the first subsection yields Assertion 2. of the proposition, whereas Assertion 1. is proved in the second and third subsections.

  \subsection*{Island borders} In this subsection we prove Assertion 2. of the proposition, namely we consider $x_0=c$ for some $(c,d)\in I_h$, and we prove that $\theta^-(c)\leq\theta^*$ (see Figure \ref{YDStep2Case2Big}). The case of $x_0=d$ and $\theta^+(d)$ is analogous by symmetry. For simplicity we denote in the following $\theta^-(c)$ by $\theta$. 
 
 We begin by considering $\theta^*>0$. Note that
 \be{eq:tan-s}
\tan(\theta)=-g'_{-}(0)\quad\text{and }\tan(\theta^*)=\frac{\sqrt{1-\beta^2}}{\beta}. 
 \ee 
  
Assume by contradiction that \be{absurdhpborders} \theta> \theta^*.\ee Then by \eqref{eq:tan-s}, we have $0<\tan(\theta^*)<-g'_{-}(0)$. We define $\psi_{\delta}$ by
$$
\psi_{\delta}(s)=\begin{cases}-\left(\frac{g'_{-}(0)+\tan(\theta^*)}{\delta}\right)s + \frac{\tan(\theta^*)}{g'_{-}(0)}+1& \text{for } \frac{\delta}{g'_{-}(0)}<s\leq0,\\
\frac{-\tan(\theta^*)}{\delta}s +\frac{\tan(\theta^*)}{g'_{-}(0)}+1& \text{for }  0<s<\delta\left(\frac{1}{\tan(\theta^*)}+\frac{1}{g'_{-}(0)}\right),\\
0 & \text{otherwise.}
\end{cases}
$$
 \begin{figure}[htp]
\begin{center}
\includegraphics[scale=0.14]{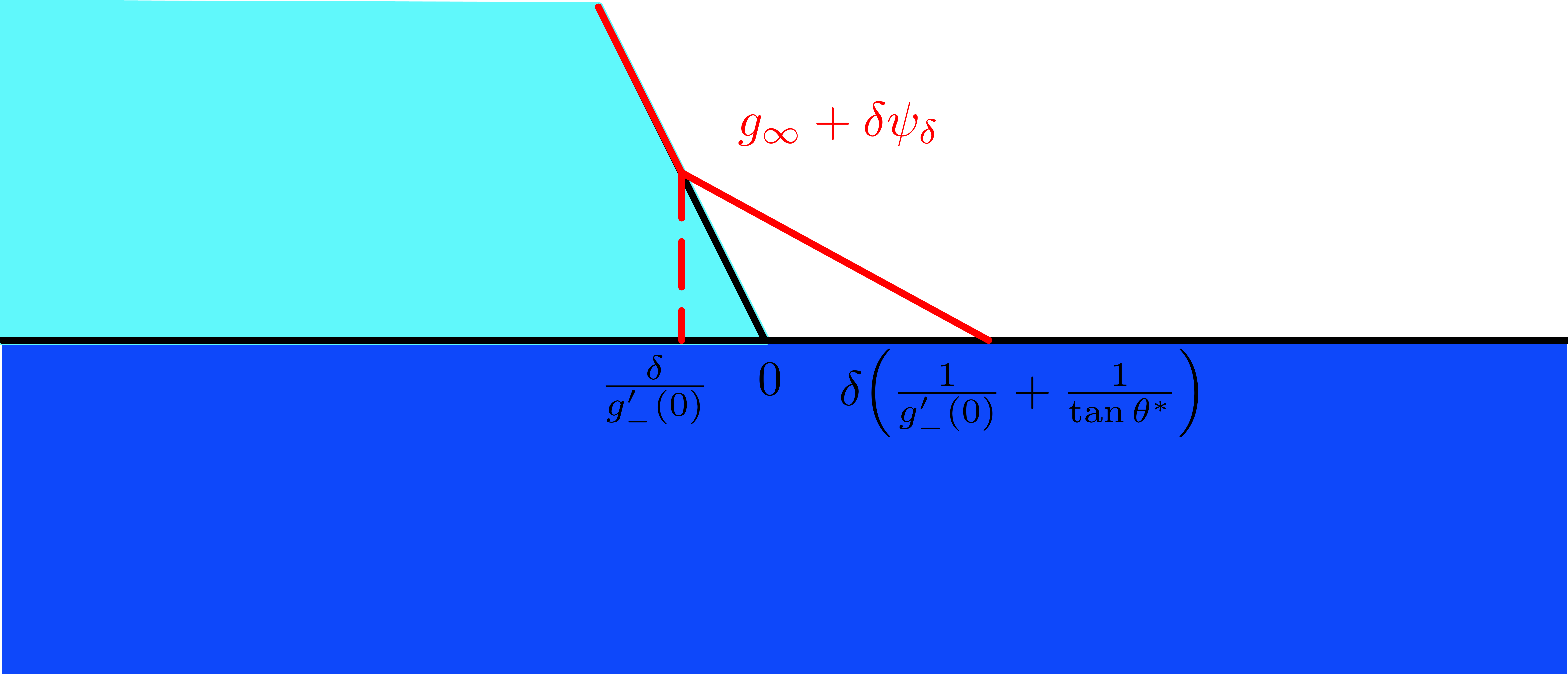}
\caption{The blow-up at island borders (and at valleys with one vanishing contact angle) is displayed. The profile of $g_{\infty}$ and the perturbation $g_{\infty}+\delta \psi_{\delta}$ in the case $\theta\geq\theta^*=\arccos{\beta}$ for island borders are highlighted in black and in red, respectively}
\label{YDStep2Case2Big}
\end{center}
\end{figure}

We observe that
$$\HH^1(\tilde{\Gamma}_{h+\delta\psi_{\delta}^n}\cap\{y=0\})-\HH^1(\tilde{\Gamma}_{h}\cap\{y=0\})=-\delta r_n\left(\frac{1}{\tan(\theta^*)}+\frac{1}{g'_{-}(0)}\right)$$
as shown in Figure \ref{YDStep2Case2Big}. We now observe that by condition ($c_1$) the map $h$ is Lipschitz in $(-a',a')$ and hence, its derivative $h'$ exists a.e. in $(-a',a')$, and $h'_-$ and $h'_+$ are, respectively, left and right continuous. These properties together with the definition of $\psi_{\delta}$ imply that  
\bas
&B_n=\frac{\gamma_f}{\delta r_n}\int_{\frac{\delta r_n}{g'_{-}(0)}+x_0}^{\delta r_n\left(\frac{1}{\tan(\theta^*)}+\frac{1}{g'_{-}(0)}\right)+x_0}\sqrt{1+\left(h'(x)+\delta\psi_{\delta}'\left(\frac{x-x_0}{r_n}\right)\right)^2}\,dx\\
&\quad-\frac{\gamma_f}{\delta r_n}\int_{\frac{\delta r_n}{g'_{-}(0)}+x_0}^{x_0}\sqrt{1+(h'(x))^2}\,dx
-\gamma_f\beta\left(\frac{1}{\tan(\theta^*)}+\frac{1}{g'_{-}(0)}\right)\\
&\quad=\frac{\gamma_f}{\delta}\int_{\frac{\delta}{g'_{-}(0)}}^{\delta\left(\frac{1}{\tan(\theta^*)}+\frac{1}{g'_{-}(0)}\right)}\sqrt{1+\left(h'(x_0+r_ns)+\delta\psi_{\delta}'(s)\right)^2}\,ds\\
&\quad-\frac{\gamma_f}{\delta}\int_{\frac{\delta}{g'_{-}(0)}}^{0}\sqrt{1+(h'(x_0+r_ns))^2}\,ds
-\gamma_f\beta\left(\frac{1}{\tan(\theta^*)}+\frac{1}{g'_{-}(0)}\right),
\end{align*}
 where in the last equality we used the change of variable \eqref{variablechange}.
 Furthermore, in view of the fact that  $h'_-(x_0+r_nz)\to g'_{-}(0)$ and $h'_+(x_0+r_nz)\to g'_{+}(0)$ as $n\to +\infty$,  the Lebesgue Dominated Convergence Theorem yields that
\begin{align}
\label{eq:lim-Bn}B_n\to -\gamma_f\beta\left(\frac{1}{\tan(\theta^*)}+\frac{1}{g'_{-}(0)}\right)+\frac{\gamma_f}{\beta\tan(\theta^*)}+\gamma_f\frac{\sqrt{1+(g'_{-}(0))^2}}{g'_{-}(0)}.
\end{align}
By \eqref{eq:el}, \eqref{eq:dn}, and Step 1, there holds
\be{eq:lim-Bn-2zero}-\beta\left(\frac{1}{\tan(\theta^*)}+\frac{1}{g'_{-}(0)}\right)+\frac{1}{\beta\tan(\theta^*)}+\frac{\sqrt{1+(g'_{-}(0))^2}}{g'_{-}(0)}\geq 0,\ee
which in turn implies
\be{eq:lim-Bn2prima}
\beta \tan(\theta^*)\sqrt{1+(\tan(\theta))^2}\leq (1-\beta^2)\tan(\theta)+\beta^2\tan(\theta^*).
\ee
Substituting \eqref{eq:tan-s} in \eqref{eq:lim-Bn2prima}, dividing by $\sqrt{1-\beta^2}$, and taking the squares of both sides of the resulting inequality, we obtain
\be{eq:lim-Bn2}\left(\beta \tan(\theta)-\sqrt{1-\beta^2}\right)^2\leq 0,
\ee
and hence, again by \eqref{eq:tan-s}, $\theta=\theta^*$ which is in contradiction with \eqref{absurdhpborders}.

Consider now the case in which $\theta^*=0$, i.e., $\beta=1$. Assume by contradiction that \be{absurdhpborderszero} \theta> \theta^*=0.\ee Then, for $\delta$ small enough, by \eqref{eq:tan-s}, we have $0=\tan(\theta^*)<\delta<-g'_{-}(0)$. We define $\psi_{\delta}$ by
$$
\psi_{\delta}(s)=\begin{cases}-\left(\frac{g'_{-}(0)+\ep_{\delta}}{\delta}\right)s + \frac{\ep_{\delta}}{g'_{-}(0)}+1& \text{for } \frac{\delta}{g'_{-}(0)}<s\leq0,\\
\frac{-\ep_{\delta}}{\delta}s +\frac{\ep_{\delta}}{g'_{-}(0)}+1& \text{for }  0<s<\delta\left(\frac{1}{\ep_{\delta}}+\frac{1}{g'_{-}(0)}\right),\\
0 & \text{otherwise},
\end{cases}
$$
where $\ep_{\delta}<<\delta$ is such that $\delta\left(\frac{1}{\ep_{\delta}}+\frac{1}{g'_{-}(0)}\right)<1$. 

The same computations as in the case $\theta^*>0$ yield
$$\HH^1(\tilde{\Gamma}_{h+\delta\psi_{\delta}^n}\cap\{y=0\})-\HH^1(\tilde{\Gamma}_{h}\cap\{y=0\})=-\delta r_n\left(\frac{1}{\ep_{\delta}}+\frac{1}{g'_{-}(0)}\right),$$
 and hence, since here $\beta=1$,
\bas
&B_n=\frac{\gamma_f}{\delta r_n}\int_{\frac{\delta r_n}{g'_{-}(0)}+x_0}^{\delta r_n\left(\frac{1}{\ep_{\delta}}+\frac{1}{g'_{-}(0)}\right)+x_0}\sqrt{1+\left(h'(x)+\delta\psi_{\delta}'\left(\frac{x-x_0}{r_n}\right)\right)^2}\,dx\\
&\quad-\frac{\gamma_f}{\delta r_n}\int_{\frac{\delta r_n}{g'_{-}(0)}+x_0}^{x_0}\sqrt{1+(h'(x))^2}\,dx
-\gamma_f\left(\frac{1}{\ep_{\delta}}+\frac{1}{g'_{-}(0)}\right)\\
&\quad=\frac{\gamma_f}{\delta}\int_{\frac{\delta}{g'_{-}(0)}}^{\delta\left(\frac{1}{\ep_{\delta}}+\frac{1}{g'_{-}(0)}\right)}\sqrt{1+\left(h'(x_0+r_ns)+\delta\psi_{\delta}'(s)\right)^2}\,ds\\
&\quad-\frac{\gamma_f}{\delta}\int_{\frac{\delta}{g'_{-}(0)}}^{0}\sqrt{1+(h'(x_0+r_ns))^2}\,ds
-\gamma_f\left(\frac{1}{\ep_{\delta}}+\frac{1}{g'_{-}(0)}\right),
\end{align*}
which in turn, by the Dominated Convergence Theorem, implies
\begin{align}
\label{eq:lim-Bnzero}B_n\to -\gamma_f\left(\frac{1}{\ep_{\delta}}+\frac{1}{g'_{-}(0)}\right)+\gamma_f\frac{\sqrt{1+\ep_{\delta}^2}}{\ep_{\delta}}+\gamma_f\frac{\sqrt{1+(g'_{-}(0))^2}}{g'_{-}(0)}.
\end{align}
Since the function $x\to \frac{\sqrt{1+x^2}}{x}-\frac{1}{x}$ is strictly increasing in $(-\infty,0)$, inequality \eqref{eq:lim-Bnzero} gives
$$0>g'_{-}(0)\geq -\ep_{\delta}.$$
By the arbitrary smallness of $\ep_{\delta}$ we conclude that $g'_{-}(0)=0$. This contradicts \eqref{absurdhpborderszero}, and completes the proof of Assertion 2. of the proposition.
 
\subsection*{Valleys with no vanishing contact angles} In this subsection we begin the proof of Assertion 1. of the proposition, namely we consider a point $x_0\in P_h$ and we prove that at least one between $g'_{-}(0)$ and $g'_{+}(0)$ is zero.
Notice that since the profile of the film is a graph we have 
 $$ g'_{-}(0)\leq0\leq g'_{+}(0).$$ 
 Assume by contradiction that
 \be{eq:differenttangents}
 g'_{-}(0)<0<g'_{+}(0),
\ee
 and define $\psi_{\delta}$ by
 $$\psi_{\delta}(s):=\begin{cases}0&\text{for }s<\frac{\delta}{g'_{-}(0)}\text{ and }s>\frac{\delta}{g'_{+}(0)}.\\
 \frac{\delta-g_{\infty}(s)}{\delta}&\text{for }s\in \left[\frac{\delta}{g'_{-}(0)},\frac{\delta}{g'_{+}(0)}\right]\end{cases}$$
 for every $s\in(-1,1)$ (see Figure \ref{YDStep2Case1}). 
 Since $\psi_{\delta}\geq0$, by \eqref{eq:B} we obtain 
 $$B_n=\frac{\gamma_f}{\delta r_n}(\HH(\tilde{\Gamma}_{h+\delta\psi_{\delta}^n}\cap\{y>0\})-\HH(\tilde{\Gamma}_{h}\cap\{y>0\})).$$
  \begin{figure}[htp]
\begin{center}
\includegraphics[scale=0.14]{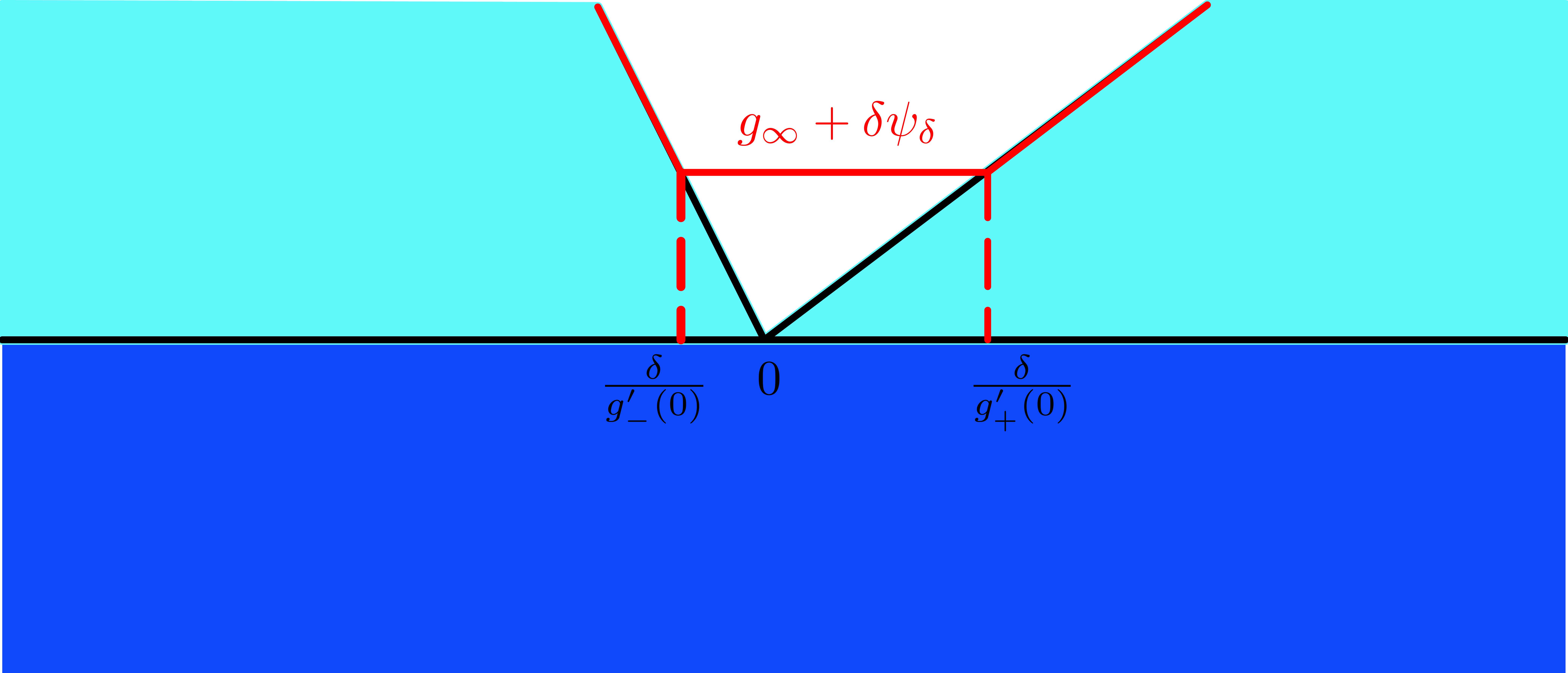}
\caption{The blow-up at a valley with no vanishing contact angle is displayed. The profile of $g_{\infty}$ and the perturbation $g_{\infty}+\delta \psi_{\delta}$ are highlighted in black and in red, respectively}
\label{YDStep2Case1}
\end{center}
\end{figure}
As before, since $h$ is Lipschitz in $(-a',a')$, the definition of $\psi_{\delta}$ implies that
 \bas
 &B_n=\frac{\gamma_f}{\delta r_n}\int_{\frac{\delta r_n}{g'_{-}(0)}+x_0}^{^{\frac{\delta r_n}{g'_{+}(0)}+x_0}}\left(\sqrt{1+\left(h'(x)+\delta\psi_{\delta}'\left(\frac{x-x_0}{r_n}\right)\right)^2}-\sqrt{1+(h'(x))^2}\right)\,dx\\
 &\quad=\frac{\gamma_f}{\delta r_n}\int_{\frac{\delta r_n}{g'_{-}(0)}+x_0}^{x_0}\left(\sqrt{1+\left(h'(x)-g'_{-}(0)\right)^2}-\sqrt{1+(h'(x))^2}\right)\,dx\\
 &\quad+\frac{\gamma_f}{\delta r_n}\int_{x_0}^{\frac{\delta r_n}{g'_{+}(0)}+x_0}\left(\sqrt{1+\left(h'(x)-g'_{+}(0)\right)^2}-\sqrt{1+(h'(x))^2}\right)\,dx\\
 &\quad=\frac{\gamma_f}{\delta}\int_{\frac{\delta}{g'_{-}(0)}}^{0}\left(\sqrt{1+\left(h'_-(x_0+r_nz)-g'_{-}(0)\right)^2}-\sqrt{1+(h'_-(x_0+r_n z))^2}\right)\,dz\\
 &\quad+\frac{\gamma_f}{\delta}\int_{0}^{\frac{\delta}{g'_{+}(0)}}\left(\sqrt{1+\left(h'_+(x_0+r_n z)-g'_{+}(0)\right)^2}-\sqrt{1+(h'_+(x_0+r_n z))^2}\right)\,dx
 \end{align*}
 where in the last equality we used the change of variable \eqref{variablechange}.
 Furthermore, in view of the fact that  $h'_-(x_0+r_nz)\to g'_{-}(0)$ and $h'_+(x_0+r_nz)\to g'_{+}(0)$ as $n\to +\infty$,  the Lebesgue Dominated Convergence Theorem yields that
  $$B_n\to\gamma_f\left(\frac{g'_{-}(0)}{1+\sqrt{1+(g'_{-}(0))^2}}-\frac{g'_{+}(0)}{1+\sqrt{1+(g'_{+}(0))^2}}\right).$$
as $n\to +\infty$.
 By \eqref{eq:el}, \eqref{eq:Ezero}, \eqref{eq:dn},  and Step 1, there holds
 \be{eq:el-lim1}\frac{g'_{-}(0)}{1+\sqrt{1+(g'_{-}(0))^2}}\geq \frac{g'_{+}(0)}{1+\sqrt{1+(g'_{+}(0))^2}}.\ee
 We observe that, setting $f(x):=\frac{x}{1+\sqrt{1+x^2}}$ for every $x\in\R$, there holds $f'(x)>0$ for every $x\in\R$. Thus \eqref{eq:el-lim1} yields that $g'_{-}(0)\geq g'_{+}(0)$ which is in contradiction with \eqref{eq:differenttangents}.

\subsection*{Valleys with one vanishing contact angle}
In this subsection we conclude the proof of Assertion 1. of the proposition. From the previous subsection it remains to prove that
if $x_0\in P_h$ is such that $\theta^+(x_0)=0$, then $\theta^-(x_0)\leq \theta^*$ (see Figure \ref{YDStep2Case2Big}). In the symmetric case, in which $x_0\in P_h$ is such that $\theta^-(x_0)=0$, analogous arguments imply that $\theta^+(x_0)\leq \theta^*$. 

Let $x_0\in P_h$ with $\theta^+(x_0)=0$. We first consider the case $\theta^*:=\arccos(\beta)>0$. Assume by contradiction that  
\be{eq:contrad-valley-one}
\theta:=\theta^-(x_0)>\theta^*.
\ee  We define $\psi_{\delta}$ as in the case of island borders, by
$$
\psi_{\delta}(s)=\begin{cases}-\left(\frac{g'_{-}(0)+\tan(\theta^*)}{\delta}\right)s + \frac{\tan(\theta^*)}{g'_{-}(0)}+1& \text{for } \frac{\delta}{g'_{-}(0)}<s\leq0,\\
\frac{-\tan(\theta^*)}{\delta}s +\frac{\tan(\theta^*)}{g'_{-}(0)}+1& \text{for }  0<s<\delta\left(\frac{1}{\tan(\theta^*)}+\frac{1}{g'_{-}(0)}\right),\\
0 & \text{otherwise.}
\end{cases}
$$
Differently from the case of island borders, we have
$$\HH^1(\tilde{\Gamma}_{h+\delta\psi_{\delta}^n}\cap\{y=0\})-\HH^1(\tilde{\Gamma}_{h}\cap\{y=0\})=0,$$
and
\bas
&B_n=\frac{\gamma_f}{\delta r_n}\int_{\frac{\delta r_n}{g'_{-}(0)}+x_0}^{\delta r_n\left(\frac{1}{\tan(\theta^*)}+\frac{1}{g'_{-}(0)}\right)+x_0}\sqrt{1+\left(h'(x)+\delta\psi_{\delta}'\left(\frac{x-x_0}{r_n}\right)\right)^2}\,dx\\
&\quad-\frac{\gamma_f}{\delta r_n}\int_{\frac{\delta r_n}{g'_{-}(0)}+x_0}^{x_0}\sqrt{1+(h'(x))^2}\,dx
-\gamma_f\left(\frac{1}{\tan(\theta^*)}+\frac{1}{g'_{-}(0)}\right)\\
&\quad=\frac{\gamma_f}{\delta}\int_{\frac{\delta}{g'_{-}(0)}}^{\delta\left(\frac{1}{\tan(\theta^*)}+\frac{1}{g'_{-}(0)}\right)}\sqrt{1+\left(h'(x_0+r_ns)+\delta\psi_{\delta}'(s)\right)^2}\,ds\\
&\quad-\frac{\gamma_f}{\delta}\int_{\frac{\delta}{g'_{-}(0)}}^{0}\sqrt{1+(h'(x_0+r_ns))^2}\,ds
-\gamma_f\left(\frac{1}{\tan(\theta^*)}+\frac{1}{g'_{-}(0)}\right).
\end{align*}
Arguing as in Step 2 in the case of island borders, by the Dominated Convergence Theorem, we obtain
$$B_n\to -\gamma_f\left(\frac{1}{\tan(\theta^*)}+\frac{1}{g'_{-}(0)}\right)+\frac{\gamma_f}{\beta\tan(\theta^*)}+\gamma_f\frac{\sqrt{1+(g'_{-}(0))^2}}{g'_{-}(0)}\geq 0.$$
Since $\beta\leq 1$, the previous inequality implies \eqref{eq:lim-Bn-2zero}, which in turn, arguing as in the case of island borders, yields $\theta=\theta^*$. This contradicts \eqref{eq:contrad-valley-one}.

Consider now the case in which $\theta^*=0$, and assume by contradiction that  
\be{eq:contrad-valley-onezero}
\theta:=\theta^-(x_0)>\theta^*,
\ee  
namely $\beta=1$.
Then, for $\delta$ small enough, by \eqref{eq:tan-s}, we have $0=\tan(\theta^*)<\delta<-g'_{-}(0)$.
We define $\psi_{\delta}$ as in the case of island borders, by
$$
\psi_{\delta}(s)=\begin{cases}-\left(\frac{g'_{-}(0)+\ep_{\delta}}{\delta}\right)s + \frac{\ep_{\delta}}{g'_{-}(0)}+1& \text{for } \frac{\delta}{g'_{-}(0)}<s\leq0,\\
\frac{-\ep_{\delta}}{\delta}s +\frac{\ep_{\delta}}{g'_{-}(0)}+1& \text{for }  0<s<\delta\left(\frac{1}{\ep_{\delta}}+\frac{1}{g'_{-}(0)}\right),\\
0 & \text{otherwise.}
\end{cases}
$$
where $\ep_{\delta}<<\delta$ is such that $\delta\left(\frac{1}{\ep_{\delta}}+\frac{1}{g'_{-}(0)}\right)<1$. Analogous computations to the case $\theta^*>0$, as well as the fact that $\beta=1$, yield the inequality
\be{eq:eq:contrad-valley-onezero}-\gamma_f\left(\frac{1}{\tan(\theta^*)}+\frac{1}{g'_{-}(0)}\right)+\gamma_f\frac{\sqrt{1+\ep_{\delta}^2}}{\ep_{\delta}}+\gamma_f\frac{\sqrt{1+(g'_{-}(0))^2}}{g'_{-}(0)}\geq 0,
\ee
which is the same relation that we obtained in \eqref{eq:lim-Bnzero}. As in Step 2, in the case of island borders with $\theta^*=0$, we deduce that $0>g'_{-}(0)\geq -\ep_{\delta},$
and, by the arbitrary smallness of $\ep_{\delta}$, that $g'_{-}(0)=0$. This contradicts \eqref{eq:contrad-valley-onezero} and completes the proof of Assertion 1.

  \noindent\textbf{Step 3} (Surface-energy convergence under conditions ($c_2$) or ($c_3$))\textbf{.} %In this step we consider conditions ($c_2$) or ($c_3$). 
  We point out that conditions ($c_2$) or ($c_3$) correspond to $z_0$ being a lower-endpoint of a connected component of $\Gamma_h^{jump}$. In this step we prove Assertion 3. of the proposition, namely we show that conditions ($c_2$) and ($c_3$) are never satisfied except when $\beta=0$.  As in the previous step we distinguish the case of island borders, of valleys with no vanishing contact angles, and of valleys with one vanishing contact angle. We only consider condition ($c_3$). The same arguments work under condition ($c_2$). 
  
 \subsection*{Jumps: Island borders} Here we prove that if $\theta^*\neq \frac{\pi}{2}$ then there are no jumps at island borders. Assume by contradiction that there exists $x_0=c$ for some $(c,d)\in I_h$ such that $\theta^-(c)= \pi/2$, and that $\theta^*\neq \frac{\pi}{2}$. 
 
 We first consider the case in which $\theta^*>0$.  
We choose $\psi_{\delta}$ such that
$$\psi_{\delta}(s):=\begin{cases}-\frac{\tan(\theta^*)}{\delta}s&\text{for }0<s<\frac{\delta}{\tan(\theta^*)},\\
0&\text{otherwise }.\end{cases}$$
We observe that
$$\HH(\tilde{\Gamma}_{h+\delta\psi_{\delta}^n}\cap\{y=0\})-\HH(\tilde{\Gamma}_{h}\cap\{y=0\})=-\frac{\delta r_n}{\tan(\theta^*)}.$$
Therefore
\bas
&B_n=\frac{\gamma_f}{\delta r_n}\int_{x_0}^{\frac{\delta r_n}{\tan(\theta^*)}+x_0}\sqrt{1+\left(h'(x)+\delta\psi_{\delta}'\left(\frac{x-x_0}{r_n}\right)\right)^2}\,dx-\frac{\gamma_f\beta}{\tan(\theta^*)}-\gamma_f\\
&\quad=\frac{\gamma_f}{\delta}\int_{0}^{\frac{\delta}{\tan(\theta^*)}}\sqrt{1+\left(h'(x_0+r_ns)+\delta\psi_{\delta}'(s)\right)^2}\,ds-\frac{\gamma_f\beta}{\tan(\theta^*)}-\gamma_f.
\end{align*}
By the Dominated Convergence Theorem, we obtain
\begin{align*}
B_n\to -\frac{\gamma_f\beta}{\tan(\theta^*)}+\frac{\gamma_f}{\beta\tan(\theta^*)}-\gamma_f.
\end{align*}
By \eqref{eq:el}, \eqref{eq:dn}, and Step 1, there holds
$$-1-\frac{\beta}{\tan(\theta^*)}+\frac{1}{\beta\tan(\theta^*)}\geq 0.$$
Thus, in view of \eqref{eq:tan-s} we conclude that
\be{eq:final-jump-ib}\sin(\theta^*)=\sqrt{1-\beta^2}\geq 1,\ee
namely, a contradiction. 

Consider now the case in which $\theta^*=0$, and choose 
$$\psi_{\delta}(s):=-s$$
for every $s\in (0,1)$. Then,
$$\HH(\tilde{\Gamma}_{h+\delta\psi_{\delta}^n}\cap\{y=0\})-\HH(\tilde{\Gamma}_{h}\cap\{y=0\})=-r_n,$$
and, since $\beta=1$,
\bas
&B_n=\frac{\gamma_f}{\delta r_n}\int_{x_0}^{r_n+x_0}\sqrt{1+\left(h'(x)+\delta\psi_{\delta}'\left(\frac{x-x_0}{r_n}\right)\right)^2}\,dx-\frac{\gamma_f}{\delta}-\gamma_f\\
&\quad=\frac{\gamma_f}{\delta}\int_{0}^{1}\sqrt{1+\left(h'(x_0+r_ns)+\delta\psi_{\delta}'(s)\right)^2}\,ds-\frac{\gamma_f}{\delta}-\gamma_f.
\end{align*}
By the Dominated Convergence Theorem, we have
\begin{align*}
B_n\to -\frac{\gamma_f}{\delta}+\gamma_f\frac{\sqrt{1+\delta^2}}{\delta}-\gamma_f.
\end{align*}
Thus, properties \eqref{eq:el}, \eqref{eq:dn}, and Step 1 imply that
$$-1-\frac{1}{\delta}+\frac{\sqrt{1+\delta^2}}{\delta}\geq 0.$$
Since $-1-\frac{1}{\delta}+\frac{\sqrt{1+\delta^2}}{\delta}< 0$ for every $\delta>0$, we reach also in this case a contradiction. 
 
\subsection*{Jumps: Valleys with no vanishing contact angles} In this subsection we prove that for every $\theta^*$ there are no jumps at valleys with no vanishing contact angles. Consider $x_0\in P_h$ and such that $\theta^-(x_0)=\frac{\pi}{2}$. We want to prove that $g'_{+}(0)=0$. Assume by contradiction that $g'_{+}(0)>0$. Let  
 $$\psi_{\delta}(s):=\begin{cases}0&\text{for }s>\frac{\delta}{g'_{+}(0)}.\\
 \frac{\delta-g_{\infty}(s)}{\delta}&\text{for }s\in \left(0,\frac{\delta}{g'_{+}(0)}\right],\end{cases}$$
 for every $s\in(0,1)$. By the definition of $\psi_{\delta}$ there holds
  $$B_n=\frac{\gamma_f}{\delta r_n}(\HH(\tilde{\Gamma}_{h+\delta\psi_{\delta}^n}\cap\{y>0\})-\HH(\tilde{\Gamma}_{h}\cap\{y>0\})).$$
 In particular, we obtain
 \bas
 &B_n=\frac{\gamma_f}{\delta r_n}\int_{x_0}^{^{\frac{\delta r_n}{g'_{+}(0)}+x_0}}\left(\sqrt{1+\left(h'(x)+\delta\psi_{\delta}'\left(\frac{x-x_0}{r_n}\right)\right)^2}-\sqrt{1+(h'(x))^2}\right)\,dx-\gamma_f\\
 &\quad=\frac{\gamma_f}{\delta r_n}\int_{x_0}^{\frac{\delta r_n}{g'_{+}(0)}+x_0}\left(\sqrt{1+\left(h'(x)-g'_{+}(0)\right)^2}-\sqrt{1+(h'(x))^2}\right)\,dx-\gamma_f\\
 &\quad=\frac{\gamma_f}{\delta}\int_{0}^{\frac{\delta}{g'_{+}(0)}}\left(\sqrt{1+\left(h'(x_0+r_n z)-g'_{+}(0)\right)^2}-\sqrt{1+(h'(x_0+r_n z))^2}\right)\,dx-\gamma_f.
 \end{align*}
By applying the Dominated Convergence Theorem we conclude that
  $$B_n\to-\gamma_f\frac{g'_{+}(0)}{1+\sqrt{1+(g'_{+}(0))^2}} -\gamma_f$$
as $n\to +\infty$. Therefore, properties \eqref{eq:el}, \eqref{eq:dn}, and Step 1 yield 
$$g'_{+}(0)\leq -1-\sqrt{1+(g'_{+}(0))^2},$$
which contradicts the non negativity of $g'_{+}(0)$.

\subsection*{Jumps: Valleys with one vanishing contact angle}
Here we prove that if $\theta^*\neq \frac{\pi}{2}$ then there are no jumps at valleys with one vanishing contact angle (and hence, by the previous subsection, at every valley). 
Assume by contradiction that $\theta^*\neq \frac{\pi}{2}$, and that there exists $x_0\in P_h$ with $\theta^-(x_0)=\frac{\pi}{2}$ and $g'_{+}(0)=0$. 

In the situation in which $\theta^*>0$ we argue choosing $\psi_{\delta}$ as in the corresponding situation in Step 3, in the case of island borders. The same computations as in that subsection yield
$$B_n\to -\frac{\gamma_f}{\tan(\theta^*)}+\frac{\gamma_f}{\beta\tan(\theta^*)}-\gamma_f\geq 0.$$
Since $\beta<1$, this implies
$$-\frac{\gamma_f\beta}{\tan(\theta^*)}+\frac{\gamma_f}{\beta\tan(\theta^*)}-\gamma_f\geq 0,$$
which in turn yields to \eqref{eq:final-jump-ib} and to a contradiction.

The situation in which $\theta^*=0$ can be dealt with exactly in the same way as in the corresponding setting in Step 3 for island borders.

%Assertion 4. follows by observing that the same arguments used for the case of valleys in Steps 2 and 3 also apply to points $(x_0,h(x_0))\in (\Gamma_h^{reg}\cup\Gamma_h^{jump})\setminus Z_h$.

\end{proof}
%\UUU Here we still need to consider the situation in which we have a cusp ``just from one side".\EEE

We are now ready to prove Theorem \ref{thm:YDlaw}.

\begin{proof}[Proof of Theorem \ref{thm:YDlaw}]
We observe that Assertion 3. of Theorem  \ref{thm:YDlaw} coincides with Assertion 3. of Proposition \ref{pro:YDlawold}. In the wetting regime $\beta=1$ also Assertion 1. of Theorem  \ref{thm:YDlaw} follows directly from Assertions 1. and 2. of Proposition  \ref{pro:YDlawold}. 
Furthermore, in the dewetting regime $\beta<1$ from Proposition \ref{pro:YDlawold} for any $(c,d)\in I_h$  and $p\in P_h$ the angles $\theta^{-}(p)$, $\theta^{+}(p)$, $\theta^{-}(c)$, and $\theta^+(d)$ are smaller or equal to $\theta^*$ (and at least one between  $\theta^{-}(p)$ and $\theta^{+}(p)$ is zero). 
It remains therefore to assume that $\beta<1$, and in turn 
\be{thetastarpos}\theta^*>0,
\ee
 and to show that for any $(c,d)\in I_h$ the angles $\theta^{-}(c)$ and $\theta^+(d)$ are not strictly smaller than $\theta^*$, and that $P_h=\emptyset$. 
 
 To this aim we observe that it is enough to show the following claim: for every $z_0=(x_0,h(x_0))\in Z_h$ which is a valley or an island border, there holds 
 $$\theta^{-}(x_0)\geq\theta^*.$$
 In fact, we already know that in the dewetting regime any $p\in P_h$ has at least a zero contact angle from Proposition \ref{pro:YDlawold}.
 
To show the claim, we argue by contradiction and we assume that there exists a point  $z_0=(x_0,h(x_0))\in Z_h$ such that 
\be{contradictionhp}\theta^{-}(x_0)<\theta^*.
\ee
 The case with $\theta^{+}(x_0)<\theta^*$ follows by symmetry.
%We notice that both these properties are a consequence of the following claim: If there exists  $z_0=(x_0,h(x_0))\in Z_h$ such that $\theta^{-}(x_0)<\theta^*$ or $\theta^{+}(x_0)<\theta^*$, then $\beta=1$. 
%To show the claim we fix a point  $z_0=(x_0,h(x_0))\in Z_h$ such that $\theta^{-}(x_0)<\theta^*$. The case with only $\theta^{+}(x_0)<\theta^*$ follows by symmetry. 
We start by defining a competitor profile function $h_{\ep}\in AP(a,b)$ by
\be{hep}
h_{\ep}(x):=
\begin{cases}
h(x) &\textrm{if $x\not\in[x_0-\ep,x_0]$,}\\
-\tan(\theta^*)(x-x_0+\ep)+h(x_0-\ep) &\textrm{if $x\in[x_0-\ep,x_0-\ep+\ell_{\ep}]$,}\\
0     &\textrm{if $x\in[x_0-\ep+\ell_{\ep},x_0]$,}
\end{cases}
\ee
for every $x\in(a,b)$ and $\ep>0$ small enough, where the quantity
\be{elleep}
\ell_{\ep}:=\frac{h(x_0-\ep)}{\tan(\theta^*)}
\ee 
is well defined owing to \eqref{thetastarpos}.
We observe that $h\geq h_{\ep}$ and that
\be{eq:YD0}
|\Omega_{h}|-|\Omega_{h_{\ep}}|\leq \int_{x_0-\ep}^{x_0} h(x) dx= \ep \int_{-1}^{0} h(x_0-\ep y) dy
\ee
by the change of variable $x=x_0+ \ep y$. Furthermore, we notice that the integral on the right-hand side of  \eqref{eq:YD0} converges to zero by the Lebesgue Dominated Convergence Theorem because $h$ is null and continuous at $x_0$. Therefore, $(u, h_{\ep})\in X$ is admissible for the penalized minimum problem \eqref{eq:penprob} for every $\ep>0$ small enough. 

From the minimality of $(u,h)$ and Proposition \ref{prop:penalization}  it follows that
\begin{align}
\label{eq:YDa}
\nonumber \mathcal{F}(u,h)&\leq \mathcal{F}(u,h_\ep) +\lambda_0 ||\Omega_{h}|-|\Omega_{h_{\ep}}||\\
&\leq \mathcal{F}(u,h_\ep) + \lambda_0 \ep \int_{-1}^{0} h(x_0-\ep y) dy,
\end{align}
where in the last inequality we again used \eqref{eq:YD0}.
By \eqref{filmenergyfinal}, \eqref{eq:beta}, \eqref{hep}, and \eqref{elleep} we obtain
\begin{align}
\label{eq:YDb}
\nonumber \mathcal{F}(u,h_\ep)&=\int_{\Omega_{h_{\ep}}}W_0(y, Eu(x,y)-E_0(y))\,dx\,dy+\int_{\tilde{\Gamma}_{h_{\ep}}}\varphi(y)\,d\HH^1\\
&\nonumber\quad+2\gamma_f\HH^1(\Gamma_{h_{\ep}}^{cut})+\gamma_{fs}(b-a)\\
&\nonumber\leq \int_{\Omega_{h}}W_0(y, Eu(x,y)-E_0(y))\,dx\,dy + \int_{\tilde{\Gamma}_{h}}\varphi(y)\,d\HH^1\\ 
&\nonumber\quad-\gamma_f  \int_{x_0-\ep}^{x_0} \sqrt{1+(h'(x))^2} dx +\gamma_f \sqrt{h^2(x_0-\ep)+\ell^2_{\ep}}+\beta\gamma_f (\ep-\ell_{\ep})\\
&\nonumber\quad+2\gamma_f\HH^1(\Gamma_{h}^{cut})+\gamma_{fs}(b-a)\\
&\nonumber= \mathcal{F}(u,h) -\gamma_f  \ep \int_{-1}^{0} \sqrt{1+(h'_-(x_0+\ep y))^2} dy\\
&\quad +\gamma_f \sqrt{h^2(x_0-\ep)+\ell^2_{\ep}}+\beta\gamma_f (\ep-\ell_{\ep}).
\end{align}
Inequalities \eqref{eq:YDa} and \eqref{eq:YDb} yield
\begin{align}
\nn 0\leq&\frac{\lambda_0}{\gamma_f}  \int_{-1}^{0} h(x_0-\ep y) dy-\int_{-1}^{0} \sqrt{1+(h'_-(x_0+\ep y))^2} dy\\
&\qquad\qquad\qquad\qquad\qquad+\frac{h(x_0-\ep)}{\ep}\frac{\sqrt{1+\tan^2\theta^*}}{\tan\theta^*}+\beta \left(1- \frac{\ell_{\ep}}{\ep}\right).\label{eq:YDc}
\end{align}
By applying again the Lebesgue Dominated Convergence Theorem together with the observation that both $h$ and $h'_-$ are left continuous at $x_0$, $h(x_0)=0$, and $h_-'(x_0)=-\tan\left(\theta^-(x_0)\right)$, we obtain that 
\begin{align}
\nn0\leq&- \sqrt{1+\tan^2\left(\theta^-(x_0)\right)}+\tan\left(\theta^-(x_0)\right)\frac{\sqrt{1+\tan^2\theta^*}}{\tan\theta^*}\\
&\qquad\qquad\qquad\qquad\qquad\qquad\qquad\qquad+\beta \left(1- \frac{\tan\left(\theta^-(x_0)\right)}{\tan\theta^*}\right).\label{eq:YDd}
\end{align}
 If $\tan\left(\theta^-(x_0)\right)=0$, inequality \eqref{eq:YDd} implies that $\beta\geq 1$, which contradicts the fact that $\beta<1$.
 
 Assume now that $\tan\left(\theta^-(x_0)\right)\neq 0$. By dividing \eqref{eq:YDd} by $\tan\left(\theta^-(x_0)\right)$, we have 
$$
0\leq- \frac{\sqrt{1+\tan^2\left(\theta^-(x_0)\right)}}{\tan\left(\theta^-(x_0)\right)}+\frac{\sqrt{1+\tan^2\theta^*}}{\tan\theta^*}+\beta \left(\frac{1}{\tan\left(\theta^-(x_0)\right)}- \frac{1}{\tan\theta^*}\right)
$$
from which we conclude that
\be{eq:YDde}
(\beta \tan\left(\theta^-(x_0)\right)-\sqrt{1-\beta^2})^2\leq 0,
\ee
in the same way as done for passing from \eqref{eq:lim-Bn} to \eqref{eq:lim-Bn2}. From \eqref{eq:YDde} it follows that
$$
\tan\left(\theta^-(x_0)\right)=\frac{\sqrt{1-\beta^2}}{\beta}=\tan\theta^*.
$$
This contradicts \eqref{contradictionhp}, and therefore the claim and the theorem follow.

\end{proof}

\section{Regularity of local minimizers}\label{sec:analiticity}

In this section we prove Theorem \ref{thm:regularity} by improving the regularity results already contained in Section \ref{sec:regularity}.  In particular the results follow from Proposition \ref{prop:loc-lip}, the decay estimate of Proposition \ref{thm:blow}, from implementing some arguments used for Proposition \ref{pro:YDlawold}, and from proving a second decay estimate which   is independent from the specific point on the graph $\Gamma_h\setminus(\Gamma_h^{cut}\cup\Gamma_h^{cusp})$ (see \eqref{eq:blow2}).

\begin{proof}[Proof of Theorem \ref{thm:regularity}]
We begin by observing that Assertions 1. and 2. are direct consequences of Proposition \ref{prop:loc-lip}. In fact, 
as pointed out in \cite[Remark 3.6]{FFLM2}, the only situation in which case (ii) of Proposition \ref{prop:loc-lip} arises is when $z_0$ is either a cusp point or the lower-end point of a vertical cut. %We define the set of \emph{singular points} $\Gamma_{h_{sing}}$ as $$\Gamma_{h_{sing}}:=\Gamma_h^{cusp}\cup\{(x,h(x)):\,x\in C(h)\},$$
Then, by combining Proposition \ref{prop:loc-lip} with a compactness argument it follows that the set $\Gamma_h^{cusp}\cup\{(x,h(x)):\,x\in C(h)\}$ where $C(h)$ is the set defined in \eqref{eq:def-S} has finite cardinality.

%Assertion 3. directly follows from Theorem \ref{thm:YDlaw} and the fact that the same argument employed in Step 3 of Proposition \ref{pro:YDlawold} in the case of valleys can be applied also to points $z_0=(x_0,h(x_0))\in \Gamma_h\setminus(\Gamma_{\rm cut}\cup\Gamma_{\rm cusp})$ with  $h(x_0)>0$ (by using Step 1 of Proposition \ref{pro:YDlawold} for $\mathbb{C}_f=\mathbb{C}_s$). This in particular yields that \PPP the graph $\Gamma_h^{reg}$ is $C^1$ at every $(x,h^-(x))\in \Gamma_h^{jump}\setminus Z_h$. \BBB
%$$
%(x,h^-(x))\in \Gamma_h^{jump}\qquad\textrm{iff}\qquad\textrm{($h^-(x)=0$ and $\beta=0$)}.
%$$

To obtain Assertion 3. we note that, by employing a similar argument to the one of Step 2 (valleys with one vanishing contact angle) of the Proof of Proposition \ref{pro:YDlawold} in the case of valleys and for the situation of $z_0=(x_0,h(x_0))\in \Gamma_h\setminus(\Gamma_h^{cut}\cup\Gamma_h^{cusp})$ with  $h(x_0)>0$ (by using Step 1 of Proposition \ref{pro:YDlawold} for $\mathbb{C}_f=\mathbb{C}_s$) we also prove that  $\Gamma_h^{reg}\cap\{y>0\}$ is $C^1$ and hence, $\Gamma_h^{reg}\setminus Y_h\in C^1$. In view of this regularity we can implement the argument used in \cite{FFLM2}, which is based on the following decay estimate: For every parameter $0<\sigma<1$ there exist a constant $C>0$ and a radius $r_0$ such that
\begin{equation}\label{eq:blow2}
\int_{B(z_0,r)\cup\Omega_h}|\nabla u|^2\,dx\,dy\leq Cr^{2\sigma},
\end{equation}
for  all $z_0\in\Gamma_h\setminus(\Gamma_h^{cut}\cup\Gamma_h^{cusp})$ and $0<r<r_0$. %Note that \eqref{eq:blow2} improves the decay estimate proved in Proposition \ref{thm:blow} because it is independent from the specific point on the graph $\Gamma_h\setminus(\Gamma_h^{cut}\cup\Gamma_h^{cusp})$. 
In view of \eqref{eq:blow2} it is possible to prove as in \cite[Theorem 3.17]{FFLM2}  that 
\be{hdeacay}\mathcal{H}^1(\Gamma_h\cap B(z_0,r))\leq C r^{2\sigma_0}\ee
for $\sigma_0\in(1/2,1)$ and $r$ small enough. We note that \eqref{hdeacay} follows by a perturbation argument which we can reproduce also in the dewetting regime. In fact by Theorem \ref{thm:YDlaw} the set $Z_h\setminus Y_h$ does not include island borders, and so the profile $h$ is only perturbed in  $\{y>0\}$. The conclusion then follows from \eqref{hdeacay} by arguing as  in the proof of Theorem 6.1 of  \cite{Bonnet} (see \cite[Proposition 6.4]{Bonnet}). 

Assertion 4. follows as in \cite[Theorem 3.19]{FFLM2} by taking special care for the case $\C_f\not=\C_s$. In this case infact, when showing that $h$ is a classical solution of the Euler Lagrange equation \eqref{eq:el-analyt-h} in $\Gamma_h^{reg}\setminus Z_h$ it is not possible to extend the argument to $\Gamma_h^{reg}\setminus Y_h$ arguing by approximation. This difficulty is due to the presence of the transmission problem. 
\end{proof}

\section*{Acknowledgements}

The authors thank the Center for Nonlinear Analysis (NSF Grant No. DMS-0635983) and the Erwin Schr\"odinger Institute (Thematic Program: Nonlinear Flows), where part of this research was carried out. P. Piovano acknowledges support from the Austrian Science Fund (FWF) project P~29681 and the fact that this work has been funded by the Vienna Science and Technology Fund (WWTF), the City of Vienna, and Berndorf Privatstiftung through Project MA16-005. E. Davoli acknowledges the support of the Austrian Science Fund (FWF) project P~27052 and of the SFB project F65 ``Taming complexity in partial differential systems". Both authors are thankful to Serge Nicaise, and Anna-Margaret S\"andig for useful comments on the topic of transmission problems.

\bibliographystyle{plain}
\bibliography{ed}
\end{document}